\documentclass{amsart}
\usepackage[latin1]{inputenc}
\usepackage{amssymb, amsmath, color, textcomp, url,tikz}
\usetikzlibrary{matrix,arrows}
\usetikzlibrary{fit}
\usetikzlibrary{positioning}
\usepackage[labelformat=empty]{caption}

\usepackage{mdwlist}

\RequirePackage{ifpdf}
\ifpdf
   \usepackage[pdftex]{hyperref}
\else
   \usepackage[hypertex]{hyperref}
\fi

\usepackage{enumerate,xspace}

\theoremstyle{plain}
\newtheorem{theorem}{Theorem}[section]
\newtheorem{cor}[theorem]{Corollary}
\newtheorem{prop}[theorem]{Proposition}
\newtheorem*{claim}{Claim}
\newtheorem{lemma}[theorem]{Lemma}

\theoremstyle{definition}
\newtheorem{remark}[theorem]{Remark}
\newtheorem{fact}[theorem]{Fact}
\newtheorem{definition}[theorem]{Definition}

\newtheorem*{notation}{Notation}

\newcommand{\nc}{\newcommand}
\nc{\Q}{\mathbb{Q}}
\nc{\C}{\mathfrak{C}}
\nc{\N}{\mathbb{N}}
\nc{\cb}{\operatorname{Cb}}
\nc{\K}{\mathcal{K}}
\nc{\tp}{\operatorname{tp}}
\nc{\stp}{\operatorname{stp}}
\nc{\Mr}{\operatorname{MR}}
\nc{\U}{\operatorname{U}}
\nc{\p}{\operatorname{p}}

\def\aclq{\mathrm{acl}^\mathrm{eq}}

\nc{\dd}{\operatorname{d}} 
\nc{\A}{\mathcal{A}} 
\nc{\Cent}{\operatorname{C}} 

\nc{\lmto}[1]{\xrightarrow[{#1}]{}}

\nc{\cox}{\mathrm{Cox}(N)} 
\nc{\coxnu}{\mathrm{Cox}^0} 
\nc{\wob}{\mathrm{Wob}} 
\nc{\s}{\mathcal{S}} 
\nc{\sr}{\s_\mathrm{R}} 
\nc{\sL}{\s_\mathrm{L}} 
\nc{\str}{\xrightarrow{\ast}} 
\nc{\inv}{^{-1}}
\DeclareMathOperator{\m}{M}
\nc{\mi}{\m_\infty}
\nc{\psn}{\mathrm{PS}_N}
\nc{\leqr}{\leq_\mathrm{r}}
\nc{\Rd}{\operatorname{R}_\mathrm{div}}
\nc{\Rkl}{\operatorname{R}_{\prec}}
\nc{\ord}{\operatorname{ord}}

\renewcommand{\iff}{if and only if\xspace}

\def\Ind#1#2{#1\setbox0=\hbox{$#1x$}\kern\wd0\hbox to 0pt{\hss$#1\mid$\hss}
\lower.9\ht0\hbox to 0pt{\hss$#1\smile$\hss}\kern\wd0}
\def\Notind#1#2{#1\setbox0=\hbox{$#1x$}\kern\wd0\hbox to
0pt{\mathchardef\nn="0236\hss$#1\nn$\kern1.4\wd0\hss}\hbox
to 0pt{\hss$#1\mid$\hss}\lower.9\ht0
\hbox to 0pt{\hss$#1\smile$\hss}\kern\wd0}
\def\ind{\mathop{\mathpalette\Ind{}}}
\def\nind{\mathop{\mathpalette\Notind{}}}

\begin{document}

\title{Ample Hierarchy}
\date{October 20, 2012. Slight revision of section 2 and 8: July 21,
  2013. More information on ranks in section 7: 22 July, 2013}

\author{Andreas Baudisch, Amador Martin-Pizarro and Martin Ziegler}

\address{Institut f\"ur Mathematik,Humboldt-Universit\"at zu Berlin,
  D-10099 Berlin, Germany
\newline
\indent Universit\'e de Lyon CNRS, Universit\'e Lyon 1,Institut
Camille Jordan UMR5208, 43 boulevard du 11 novembre 1918, F--69622
Villeurbanne Cedex, France
\newline
\indent Mathematisches Institut, Albert-Ludwigs-Universit\"at
Freiburg, D-79104 Freiburg, Germany}
\email{baudisch@mathematik.hu-berlin.de}
\email{pizarro@math.univ-lyon1.fr} \email{ziegler@uni-freiburg.de}
\thanks{The second author conducted research with support of programme
  ANR-09-BLAN-0047 Modig as well as an Alexander von Humboldt-Stiftung
  Forschungsstipendium f\"ur erfahrene Wissenschaftler 3.3-FRA/1137360
  STP} \keywords{Model Theory, CM-triviality}

\subjclass{03C45}

\begin{abstract}
  The ample hierarchy of geometries of stables theories is strict. We
  generalise the construction of the free pseudospace to higher
  dimensions and show that the $n$-dimensional free pseudospace is
  $\omega$-stable $n$-ample yet not $(n+1)$-ample.  In particular, the
  free pseudospace is not $3$-ample. A thorough study of forking is
  conducted and an explicit description of canonical bases is
  exhibited.
\end{abstract}

\maketitle
\tableofcontents

\section{Introduction}

Morley's renowned categoricity theorem \cite{Mo65} described any
model of an uncountably categorical theory in terms of basic
foundational bricks, so-called strongly minimal sets. A long-standing
conjecture aimed to understand the geometry of a strongly minimal set
in terms of three archetypal examples: a trivial set, a vector space
over a division ring and an irreducible curve over an algebraically
closed field. The conjecture was proven wrong \cite{Hr93} by obtaining
in a clever fashion a non-trivial strongly minimal set which does not
interpret a group. In particular, Hrushovski's new strongly minimal
set does not interpret any infinite field, which follows from the fact
that the obtained structure is CM-trivial. Recall that CM-triviality
is a generalisation of 1-basedness and it prohibits a certain
point-line-plane configuration which is present in Euclidian
geometry. The simplest example of a CM-trivial theory that is not
1-based is the
\emph{free pseudoplane:} an infinite forest with infinite branching at
every node. CM-trivial theories are rather rigid and in particular
definable groups of finite Morley rank are nilpotent-by-finite
\cite{Pi95}.

Taking the pseudoplane as a guideline, a non CM-trivial
$\omega$-stable theory which does not interpret an infinite field
was constructed in a pure combinatorial way \cite{BP00}. The
structure so obtained is of infinite rank, and it remains still open
whether the construction could be modified to produce one of finite
Morley rank. In \cite{Pi00, Ev03} a whole hierarchy
of new geometries (called $n$-ample) was exhibited, infinite fields
being at the top of the classfication. Evans suggested that his
construction could be used to show that the hierarchy is strict,
though no proof was given.

The goal of this article is to generalise the
aforementioned construction to higher dimensions in order to show that
the $N$-dimensional pseudospace is $N$-ample yet not $(N+1)$-ample,
showing therefore that the ample hierarchy is proper.
After a thorough study of the pseudospace, we were able to simplify
the combinatorics behind the original construction. In particular, we
characterise non-forking and give explicit descriptions of canonical
basis of finitary types over certain substructures. Moreover, we
show that the theory of the pseudospace has weak elimination of
imaginaries.

Tent has obtained the same result \cite{kT11}
independently; however, we present a different
construction and axiomatisation of the free pseudospace for higer
dimensions. We are indebted to her as she pointed out that the
prime model of the $2$-dimensional free pseudospace could be seen as a
building. We would like to express our gratitude to Yoneda for a
careful reading of a first version of this work.

\section{Ample concepts}\label{S:AMT}
Throughout this article, we assume a certain knowledge of stability
theory, in particular nonforking and canonical bases. We refer the
reader to \cite{TZ12} for a gentle and careful explanation of these
notions. All throughout this article, we work inside a sufficiently
saturated model of a first-order theory $T$ and all sets are small
subsets of it.

We first state a fact, which we believe is common knowledge, that
will be used repeatedly.
\begin{fact}\label{F:Inter}
Given a stable theory $T$ and sets $A$, $B$, $C$ and $D$, if
$\aclq(B)\cap\aclq(C)=\aclq(A)$ and $D\ind_A BC$,  then
$$ \aclq(DB)\cap \aclq(DC) =\aclq(DA).$$
\end{fact}

\begin{proof}
 In order to show that $ \aclq(DB)\cap \aclq(DC) \subset \aclq(DA),$
pick an element $e$ in $ \aclq(DB)\cap \aclq(DC) $. The
independence $D\ind_A BC$ yields that $\cb(De/BC)$ lies in
$\aclq(B)\cap\aclq(C)=\aclq(A)$, so $e$ lies in  $\aclq(DA)$.
\end{proof}

Recall now the definition of CM-triviality and $n$-ampleness
\cite{Pi00,Ev03}.

\begin{definition}\label{D:CM} Let $T$ be a stable theory.

The theory $T$ is  \emph{$1$-based} if for every pair of
 algebraically closed (in $T^\mathrm{eq}$) subsets $A\subset B$ and
 every real tuple $c$, we have that $\cb(c/A)$ is algebraic over
 $\cb(c/B)$. Equivalently, for every algebraically closed set $A$ (in
$T^\mathrm{eq}$) and every real tuple $c$, the canonical base
$\cb(c/A)$ is algebraic over $c$.

 The theory $T$ is \emph{CM-trivial} if for every pair of
algebraically closed (in $T^\mathrm{eq}$) subsets  $A\subset B$ and
every real tuple $c$, if $\aclq(Ac) \cap B = A$, then $\cb(c/A)$ is
algebraic over $\cb(c/B)$.

 The theory $T$ is called \emph{$n$-ample} if there are $n+1$ real
tuples satisfying the following conditions (possibly working over
parameters):
\begin{enumerate}
\item  $\aclq(a_0,\ldots,a_i)\cap\aclq(a_0,\ldots,a_{i-1},a_{i+1})=
\aclq(a_0,\ldots,a_{i-1})$ for every $0\leq i<n$,
\item $a_{i+1} \ind_{a_i} a_0,\ldots, a_{i-1}$ for every $1\leq i<n$,
\item $a_n \nind a_0$.
\end{enumerate}

\end{definition}

By inductively choosing models $M_i\supset a_i$ such that

$$ M_i \ind_{a_i} M_0,\ldots,M_{i-1},a_{i+1},\ldots,a_n,$$

\noindent Fact \ref{F:Inter} allows us to deduce the following,
which was already remarked in \cite[Corollary 2.5]{Pi95} in the case
of CM-triviality.

\begin{remark}\label{R:Mod}
In the definition of $n$-ampleness, we can replace all tuples by
models.
\end{remark}

\begin{cor}\label{C:ampleEQ}
 A stable theory $T$ is $n$-ample if and only if $T^\mathrm{eq}$ is.
\end{cor}

Clearly, every $1$-based theory is CM-trivial. Furthermore, a theory
is $1$-based if and only if it is not $1$-ample; it is CM-trivial if
and only if it is not $2$-ample \cite{Pi00}. Also, to be $n$-ample
implies $(n-1)$-ampleness: by construction, if $a_0,\ldots,a_n$
witness that $T$ is $n$-ample, the sequence $a_0,\ldots,a_{n-1}$
witnesses that $T$ is $(n-1)$-ample. In order to see this, we need
only show that
$$a_{n-1} \nind a_0,$$
which follows from $$a_n \nind a_0$$ and
$$a_{n} \ind_{a_{n-1}} a_0,$$ by transitivity.

In order to prove that the $N$-dimensional free pseudospace is not
$(N+1)$-ample, we need only consider some of the consequences
from the conditions listed above. Therefore, we will isolate such
conditions for Section \ref{S:Beweis}.
\begin{remark}\label{R:prof}
 If the (possibly infinite) tuples $a_0,\ldots,a_n$ witness that $T$
is $n$-ample, they satisfy the following conditions:
\renewcommand{\theenumi}{\alph{enumi}}
\begin{enumerate}

\item $a_n \ind_{a_i} a_{i-1}$ for every $1\leq i<n$.
\item  $\aclq(a_i,a_{i+1})\cap\aclq(a_i,a_n)= \aclq(a_i)$ for every
$0\leq i<n-1$.
\item  $a_n \nind\limits_{\aclq(a_i)\cap \aclq(a_{i+1})} a_i$
for every $0\leq i<n-1$.
\end{enumerate}
\end{remark}
If the tuples $a_0,\ldots,a_n$ witness that
$T$ is $n$-ample over some set of parameters $A$, by adding all
elements of $A$ to each of the tuples, then we may assume that all the
conditions  hold with $A=\emptyset$.

\begin{proof}
Let $a_0,\ldots,a_n$ witness that $T$ is $n$-ample.

First, note that $\aclq(a_1)\cap \aclq(a_2)\subset \aclq(a_0)$ by
property $(1)$. For $i\leq 2$, the set  $\aclq(a_i)\cap
\aclq(a_{i+1})$ is contained in $\aclq(a_i)\cap
\aclq(a_0,\ldots,a_{i-1})$ again by $(1)$. Now, condition $(2)$
implies that $\aclq(a_i)\cap \aclq(a_0,\ldots,a_{i-1})$ is a subset of
$\aclq(a_i)\cap \aclq(a_{i-1})$. By induction, we have that
$$\aclq(a_i)\cap \aclq(a_{i+1})\subset \aclq(a_0).$$

The independence $a_n \ind_{a_i} a_{i-1}$ follows directly
from property $(2)$ and yields $(a)$. Since $a_{n} \ind_{a_{i+2}}
a_0,\ldots, a_{i+1}$, we have that $$a_n \ind_{a_i,a_{i+2}} a_{i+1}.$$
Hence, $$\aclq(a_i,a_{i+1})\cap\aclq(a_i,a_n)\subset \aclq(a_i,a_{i+1}
)\cap\aclq(a_i ,a_{i+2}),$$ and thus in $\aclq(a_0,\ldots,a_i)$ by
$(1)$. Since $$a_{i+1}\ind_{a_i}
a_0,\ldots,a_{i-1},$$ \noindent we get $(b)$.

If $$a_n \ind\limits_{\aclq(a_i)\cap \aclq(a_{i+1})} a_i $$
\noindent for some $0\leq i<n-1$, then $i > 0$ by $(3)$. Since
$a_n\ind_{a_i}
a_0,\ldots,a_{i-1}$, transitivity gives that
$$a_n\ind_{\aclq(a_i)\cap \aclq(a_{i+1})}
a_0,\ldots,a_i.$$
Thus, we obtain the independence $a_n\ind_{a_0}
a_0,\ldots,a_i$  and in particular $a_n\ind_{a_0}
a_1$. Since $a_n\ind_{a_1} a_0$ by $(2)$ and $\aclq(a_0)\cap \aclq(
a_1)=\emptyset$ by $(1)$, this implies that $$a_n\ind a_0,$$
which contradicts $(3)$.
\end{proof}

In \cite{BMPW12}, a weakening of CM-triviality was introduced, following the spirit of \cite{KoPi06}, where some of the
consequences for definable groups in $1$-based theories were extended
to type-definable groups in theories with  the
\emph{Canonical Base Property}. For the purpose of this article, we
extend the definition to all values of $n$. However, we
 do not know  of any definability properties for
groups that may follow from the general definition.

Let $\Sigma$ be an $\emptyset$-invariant family of partial types. Recall that a type $p$ over $A$ is
\emph{internal to $\Sigma$}, or \emph{$\Sigma$-internal}, if for every realisation $a$ of $p$ there is some
superset $B\supset A$ with $a\ind_A B$, and realisations $b_1,\ldots,b_r$ of types in $\Sigma$ based on $B$ such
that $a$ is definable over $B,b_1,\ldots,b_r$. If we replace definable by algebraic, then we say that $p$ is
\emph{almost internal to $\Sigma$} or \emph{almost $\Sigma$-internal}.

\begin{definition}\label{D:tight}

A stable theory $T$ is called \emph{$n$-tight} (possibly working over parameters) with respect to the family $\Sigma$
if, whenever there are $n+1$ real tuples $a_0,\ldots,a_n$ satisfying
the
following conditions:
\begin{enumerate}
 \item  $\aclq(a_0,\ldots,a_i)\cap\aclq(a_0,\ldots,a_{i-1},a_{i+1})=
\aclq(a_0,\ldots,a_{i-1})$ for every $0\leq i<n$.
\item $a_{i+1} \ind_{a_i} a_0,\ldots, a_{i-1}$ for every $1\leq i<n$,
\end{enumerate}
then $\cb(a_n/a_0)$ is almost $\Sigma$-internal over $a_1$.
\end{definition}

\begin{remark}\label{R:samedef}
 As before, we may assume that all tuples are models. In particular,
the theory $T$ is $n$-tight if and only if $T^\mathrm{eq}$ is.

 A theory $T$ is $2$-tight with respect to $\Sigma$ if for every
pair of sets $A\subset B$ and every tuple $c$, if $\aclq(Ac)\cap
\aclq(B)=\aclq(A)$, then $\cb(c/A)$ is almost
$\Sigma$-internal over $\cb(c/B)$ . In particular, this
notion agrees with \cite[Definition 3.1]{BMPW12}

If $T$ is not $n$-ample, it is $n$-tight with respect to any family
$\Sigma$. Furthermore, if $T$ is $(n-1)$-tight, it is
$n$-tight.
\end{remark}

\begin{proof}

The equivalence between both definitions is a standard reformulation
 by setting $a_0=A$,
$a_1=\cb(c/B)$ and $a_2=c$ for one direction (working over
$\aclq(a_0)\cap\aclq(a_1)$), and $A=a_0$, $B=a_0\cup \cb(a_2/a_1)$ and
$c=a_2$ for the other.

If $T$ is not $n$-ample, it is clearly $n$-tight, since algebraic
types are always almost $\Sigma$-internal for any $\Sigma$.

Suppose now that $T$ is $(n-1)$-tight, and consider $n+1$ tuples
 $a_0,\ldots,a_n$ witnessing $(1)$ and $(2)$. So
do $a_0,\ldots,a_{n-1}$ as well. Hence, the canonical base
$\cb(a_{n-1}/a_0)$ is almost $\Sigma$-internal over $a_1$.

Since $a_n\ind_{a_{n-1}} a_0$, it follows by transitivity that
$\cb(a_n/a_0)$ is algebraic over $\cb(a_{n-1}/a_0)$ and
therefore the former is also almost $\Sigma$-internal over $a_1$.

\end{proof}

In this article, we will show that the free $N$-dimensional pseudospace is $N$-ample yet not
$(N+1)$-ample. Furthermore, if $N\geq 2$, it is $N$-tight with respect
to the family of Lascar rank $1$ types.

\section{Fra\"iss\'e Limits}\label{S:Fraisse}

The results in this section were obtained by the third author in an 
unpublished note \cite{mZ11} (in a slightly more general context). 
We include them here for the sake of completeness.

Throughout this section, let $\K$ denote a class of structures closed
under isomorphisms in a fixed language $\mathrm{L}$. We assume that 
the empty structure $0$ is in $\K$. Furthermore, a class
$\s$ of embeddings between elements of $\K$ is given, called 
\emph{strong embeddings}, containing all isomorphisms and closed under
composition. We also assume that the empty map $0\to A$ is
in $\s$ for every $A\in K$. 

We call a substructure $A$ of $B$ \emph{strong} if the inclusion map
is in $\s$. We denote this by $A\leq B$.

\begin{definition}
  Given an infinite cardinal $\kappa$, an increasing chain of strong
substructures $\{A_i \}_{i < \kappa}$ is \emph{rich} if, for all
  $i<\kappa$ and all strong $f:A_i\to B$, there is some $i\leq
j<\kappa i$ and a  strong $g:B\to A_j$ such that $gf:A_i\to A_j$ is
the inclusion map.

  A \emph{Fra\"iss\'e limit} of $(\K,\s)$ of length $\kappa$ is the
union of a rich sequence of length $\kappa$.
\end{definition}

\begin{theorem}\label{T:Fraisselimes}
  Suppose $(\K,\s)$ satisfies the following conditions:
  \begin{enumerate}
  \item There are at most countably many isomorphism types in $\K$.
  \item For each $A$ and $B$ in $\K$, there are at most countably many
strong embeddings $A\to B$.
  \item $\K$ has the amalgamation property with respect to strong
    embeddings.
  \end{enumerate}
  Then rich sequences exist and all Fra\"iss\'e limits of the same
length are  isomorphic.
\end{theorem}

The existence of rich sequences is easy to show. The uniqueness for
 countable Fra\"iss\'e limits will follow from the next lemma. For
that, let us say that $A$ is \emph{r-strong in a Fra\"iss\'e limit
$M$}, denoted by $A\leqr M$, if $M$ is the union of a rich sequence
starting with $A$.
\begin{lemma}\label{L:fraisse}
  A Fra\"iss\'e limit $M$ has the following properties:
\renewcommand{\theenumi}{\alph{enumi}}
\begin{enumerate}
\item\label{L:fraisse:null} $\emptyset\leqr M$
\item\label{L:fraisse:rich} for every finite $A\leqr M$ and every $B$
  in $\K$ such that $A\leq B$, there is an r-strong subset $B'$ of $M$
  containing $A$ and isomorphic to $B$ over $A$.
  \end{enumerate}
\end{lemma}

\begin{proof}
  We observe first that if $A_0\leq A_1\leq\ldots$ is a rich sequence
  and $B\leq A_0$, then the sequence $B\leq A_0\leq A_1\leq\ldots$ is
  also rich. This implies $(\ref{L:fraisse:null})$. For
  $(\ref{L:fraisse:rich})$, choose a rich sequence
  $A=A_0\leq A_1\leq\ldots$ with union $M$. If $B\geq A$ is
given,  there exists, by richness, some index $j$ and 
$B'\leq  A_j$ isomorphic to $B$ over $A$. The set $B'$ is r-strong
in $M$, since the subsequence $B'\leq A_j\leq A_{j+1}\leq\ldots$ is
again rich.
\end{proof}

The lemma implies that countable Fra\"iss\'e limits are isomorphic
by a standard back-and-forth argument: given two Fra\"iss\'e
limits $M$ and $M'$  with rich sequences $A_0\leq A_1\leq\ldots$ and
$A'_0\leq A'_1\leq\ldots$\,, consider an isomorphism $B\to B'$, where
$B$ is strong in $A_i$ and $B'$ is strong in $A'_i$. Then there is an
extension to an isomorphism $C\to C'$ such that $A_i\leq C\leq A_j$
and $A'_i\leq C'\leq A'_j$ for some $j>i$. This results in an ascending
sequence of isomorphisms whose union yields an isomorphism $M\to
M'$.

\begin{cor}
  Assume that $M$ and $M'$ are Fra\"iss\'e limits of the same length.
Given sets $B\leqr M$ and $B'\leqr M'$, every isomorphism $B\to B'$
extends  to an isomorphism $M\to M'$.
\end{cor}
The convention that $\s$ is contains all isomorphisms and is closed
under composition represents no obstacle, thanks to the following
easy remark.

\begin{remark}\label{R:composition}
  Let $\s$ be a set of embeddings between elements of $\K$ with the
  amalgamation property.  The closure of $\s$
together with all isomorphisms under composition has again
the  amalgamation property.
\end{remark}


\section{The free pseudospace}\label{S:Construction}

In this section, we will construct and axiomatise the $N$-dimensional
free pseudo\-space, which is a generalisation of \cite{BP00}, based on
the free pseudoplane. An alternative axiomatisation, in terms of
\emph{flags}, may be found in \cite{BMPZ12}.

\begin{remark}\label{E:Psplane}
Recall that the (free) pseudoplane is a bicolored graph with infinite
branching and no loops. These elementary properties describe a
complete $\omega$-stable theory of Morley rank $\omega$.

Quantifier elimination is obtained after adding
the collection of binary predicates:
\[ d_n(x,y) \,\Longleftrightarrow \text{ the distance between $x$ and
  $y$ is exactly $n$.}\]

In particular, since there are no loops, the set $d_1(x,a)$ is
strongly minimal. Morley rank for
this theory is additive and agrees with Lascar rank.  Given the type
of an element $c$ over an algebraically closed set $A$, its canonical
base $\cb(c/A)$ is the unique point $a$ in $A$ whose distance to $c$
is smallest possible (or empty if there is no path between $c$ and
$A$).  It follows that the theory has weak elimination of imaginaries
and is moreover  $CM$-trivial but not $1$-based.
\end{remark}

The idea behind the construction of the free pseudospace \cite{BP00}
is to take a free pseudoplane, whose vertices of one color are called
\emph{planes} and vertices of the other are referred to as
\emph{lines}, and on each line put an infinite set of \emph{points},
such that, for each plane, the lines which are incident with it,
together with the points on them form again a free pseudoplane.
Nevertheless, the actual construction was rather combinatorial and
therefore less intuitive. Instead, our approach consists in building a
model out of some basic operations and study the complete theory of
such a structure, in order to show that it agrees with the free
pseudospace in  \cite{BP00} for dimension $N = 2$.

\begin{definition}
  For $N\geq 1$, a \emph{colored $N$-space} $A$ is a colored graph
  with colors (or \emph{levels}) $\A_0$,\ldots,$\A_N$ such that an
  element in $\A_i$ can only be linked to vertices in
  $\A_{i-1}\cup\A_{i+1}$. We will furthermore consider two (invisible)
  levels $\A_{-1}$ and $\A_{N+1}$, consisting of a single
  \emph{imaginary} element $a_{-1}$ and $a_{N+1}$ respectively, which
  are connected to all vertices in $\A_0$ and $\A_N$
  respectively. Given such a graph $A$ and a subset $s$ of
  $\{0,\cdots,N\}$, we set
  \[\A_s(A)=\bigcup\limits_{i\in s} \A_i(A).\] Given $x$ and $y$ in
  $\A_s(A)$, its distance in $\A_s(A)$ is denoted by $\dd^A_s(x,y)$.

  Given a colored $N$-space $A$ and vertices $a$ in $\A_l(A)$ and $b\in
  A_r(A)$, we say that $b$ \emph{lies over} $a$ (or $a$ \emph{lies
    beneath} $b$) if $l<r$ and there is a path of the form
  $a=a_l,a_{l+1},\dotsc,a_r=b$. Note that $a_k$ must be in $\A_k(A)$.
  By convention, the point $a_{N+1}$ lies over all other vertices
  (including $a_{-1}$) and  $a_{-1}$ lies beneath all other
  vertices.

  With $A$, $a$ and $b$ as above, we denote by $A_a$ the subgraph of
  $A$ consisting of all the elements of $A$ lying over $a$. Similarly
  $A^b$ denotes the subgraph of all the elements lying beneath
  $b$. The subgraph $A_a^b=(A_a)^b$ consists of all the elements of
  $A$ lying between $a$ and $b$, if $a$ lies beneath $b$.
\end{definition}

Observe that, after a suitable renumbering of levels, the subgraph
$A_a$ becomes a colored $(N-l-1)$--space, whereas $A^b$ becomes a
colored $(r-1)$--space and $A_a^b$ a colored $(r-l-2)$--space.

\begin{notation}
Intervals are assumed to be non-empty
\end{notation}

\begin{definition}\label{D:alpha}
  Given an interval $s=(l_s,r_s)$ (where $-1$ and $N+1$ are possible
  values) in $\{0,\cdots,N\}$ and a colored $N$-space $A$ with two
  distinguished vertices $a_{l_s}$ in $\A_{l_s}(A)$ beneath $a_{r_s}$ in
  $\A_{r_s}(A)$, we say that $B=A\cup\{b_i\mid i\in s\}$ with
  $b_i\in\A_i(B)$ is obtained from $A$ by applying \emph{the
    operation} $\alpha_s$ on $(a_{l_s},a_{r_s})$ if
  \renewcommand{\theenumi}{\alph{enumi}}
  \begin{enumerate}
  \item The sequence $a_{l_s},b_{l_s+1},\dotsc,b_{r_s-1},a_{r_s}$ is a path in $B$.
  \item $B$ has no new edges besides the aforementioned (and those of $A$).
  \end{enumerate}

  \noindent If either $l_s=-1$ or $r_s=N+1$, then 
  $a_{l_s}$ lies automatically beneath $a_{r_s}$.
\end{definition}

The $N$-dimensional pseudospace will now be obtained by iterating
countably many times all operations $\alpha_s$ for $s$ varying over
all intervals in $[0,N]$.  Clearly, we have the following.

\begin{remark}\label{R:amalg}
  If both $B_1$ and $B_2$ are obtained from $A$ by applying
  respectively $\alpha_{s_1}$ and $\alpha_{s_2}$, then the
  graph-theoretic amalgam $C=B_1\otimes_A B_2$ is obtained by applying
  $\alpha_{s_1}$ to $B_2$ and $\alpha_2$ to $B_1$.
\end{remark}
\begin{definition}
  Given two colored $N$-spaces $A$ and $B$, we say that $A$ a
\emph{strong subspace} of $B$ if $A$ is a subgraph of $B$ and $B$ can
be obtained from $A$ by a (possibly infinite) sequence of operations
$\alpha_s$ for varying $s$.  We  denote this by $A\leq B$.
\end{definition}
A \emph{strong embedding} $A\to B$ is an isomorphism of $A$
with a strong subspace of $B$.  Let $\K_\infty$ be the class of all
finite colored $N$-spaces $A$ with $\emptyset \leq A$. By the
last remark and Remark \ref{R:composition}, the class
$\K_\infty$ has the amalgamation property with respect to  strong
embeddings . Clearly, there are only countably may isomorphism types
in $\K_\infty$ and only finitely many maps between two structures of
$\K_\infty$.  We can consider the subclass $\K_0$, where by a
$0$-strong embedding we only allow operations $\alpha_s$, for
singleton $s$. Again, the class $\K_0$ has the amalgamation property.

By Theorem \ref{T:Fraisselimes}, we define the following structures:
\begin{definition}\label{D:M_infty}
  Let $\m^N_\infty$ be the Fra\"iss\'e limit of $\K_\infty$ with
  strong embeddings and $\m^N_0$ be the Fra\"iss\'e limit of $\K_0$
  with $0$-strong embeddings, starting from a given (fixed) path
  $a_0-\ldots-a_N$, where $a_i\in\A_i$.
\end{definition}

We will drop the superindex $N$ in $\mi^N$ or $\m_0^N$
 when they are clear from the context.

In particular, the structure $\m^2_0$ so obtained agrees with the
prime model constructed in \cite{BP00}, as Theorem \ref{T:axiome}
will show.

\begin{remark}\label{R:lokal_operation}
  Let $p$ be either $0$ or $\infty$. Consider $a$ in $\A_l(\m_p^N)$
  and $b$ be in $\A_r(\m_p^N)$ lying over $a$.  Then,
  $$(\m_p^N)_a\cong \m_p^{N-l-1},$$
  $$(\m_p^N)^b\cong \m_p^{r-1},$$ $$(\m_p^N)_a^b\cong
  \m_p^{r-l-2}.$$
  \noindent Furthermore, given $-1\leq l<r\leq N+1$, we have that
  $\A_{[l,r]}(\m_p^N)\cong \m_p^{r-l-1}$.
\end{remark}
\begin{proof}
 Given a colored $N$-space $M$ and corresponding vertices $a$ and $b$,
 every operation in $M_a$ can be extended to an operation on
 $M$. Moreover, if an operation on $M$ has no meaning restricted to
 $M_a$, then $M_a$ does not change. The other statements can be proved
 in a similar fashion.
\end{proof}
We will now introduce a notion, \emph{simply connectedness}, which
traditionally implies path-connectedness topologically. Despite this
abuse of notation, we will use this term since it implies that loops
are not punctured (\emph{cf.} Remark \ref{R:simply}(\ref{R:kreis})
and Corollary \ref{C:geschl_flaggenweg}).

\begin{definition}\label{D:simplyconnected}
 A colored $N$-space $M$ is simply connected if, whenever we are
given $l<r$ in $[-1,N+1]$, an
interval $t\subset [l,r]$, vertices  $a$ in $\A_l(M)$  beneath $b$ in
$\A_r(M)$ and $x$ and $y$ in  $\A_t(M)$ lying between $a$ and
$b$ which are $t$-connected by a path of length $k$ not passing
through $a$ nor $b$, then there is a path in $\A_t(M)$ of length at
most $k$ connecting $x$ and $y$ such that every vertex in the path lies
between $a$ and $b$.
\end{definition}

Note that simply connectedness is an empty condition for $l=-1$
and $r=N+1$.

\begin{remark}\label{R:simply}
 Let $M$ be a simply connected connected colored $N$-space. The
following hold:
\begin{enumerate}
\item \label{R:baum} The subgraph $\A_{[l,l+1]}(M)$ has no closed
paths with no repetitions.
\item\label{R:kreis} In a closed path $P$ in $\A_{[l,r]}(M)$,
all elements in $P\cap \A_{[l,r)}$ are connected (in 
$\A_{[l,r)}(M)$). Likewise for the dual statement.
\end{enumerate}

\end{remark}
\begin{proof}
 For $(\ref{R:baum})$, set  $r=N+1$, $l=l$ and take $t=[l,l+1]$ in the
definition of simply connectedness.

For $(\ref{R:kreis})$, given $x$ and $y$ in $P\cap \A_{[l,r)}$, if
they are connected using an arch of $P$ in $\A_{[l,r)}(M)$, there is
nothing to prove. Otherwise, replace successively every occurrence of
a vertex $z$ in $P\cap \A_r(M)\cap P$ by a subpath in $\A_{[l,r)}(M)$
connecting the immediate neighbours of $z$ in $P$.
\end{proof}

As the following Lemma shows, simply connectedness is preserved under
application of the operations $\alpha_s$'s, 

\begin{lemma}\label{L:sc-oper}
 Let $A$ be a simply connected colored $N$-space. If $B$ is obtained
from $A$ by applying $\alpha_s$ on $(a_{l_s},a_{r_s})$,
then $B$ is simply connected as well.
\end{lemma}

\begin{proof}
By hypothesis, the set $B$ equals $A\cup S_B$, where $S_B$ is the
path
 $$a_{l_s},b_{l_s+1},\dotsc,b_{r_s-1},a_{r_s}.$$
 Let now $t\subset [l,r]$ be given, as well as  $a$
in $\A_l$  beneath $b$ in $\A_r$ and vertices $x$
and $y$ in  $\A_t$ lying between $a$ and $b$ connected by a  path $P$
in $\A_t(B)$ of length $k$. We consider the
following cases:
\renewcommand{\theenumi}{\alph{enumi}}
\renewcommand{\theenumii}{\roman{enumii}}
  \begin{enumerate}
  \item Both $a$ and $b$ lie in $B\setminus A$. Take the direct
path between $x$ and $y$.
  \item Both $a$ and $b$ lie in $A$. We consider the following
mutually exclusive subcases:
    \begin{enumerate}
    \item Both $x$ and $y$ lie in $A$: We can replace all repetitions
in $P$ to transform it into a path fully contained in $A$ of length
at most $k$. Since $A$ is simply connected, the result follows.
      \item Both $x$ and $y$ lie in $S_B$. Again, take the direct
path between $x$ and $y$.
    \item Exactly one vertex, say $y$, lies in $A$. The path $P$
must contain either $a_{l_s}$ or $a_{r_s}$. Suppose that $P$ contains
$a_{r_s}$. Hence, we can decompose $P$ into the direct connection
(which lies between $a$ and $b$) from $x$ to $a_{r_s}$  and a path
$P'$ in $\A_t(A)$ from $a_{r_s}$ to $y$.  As $A$ is simply connected,
we obtain a path in $\A_t(A)$ between $a$ and $b$
connecting $y$ and $a_{r_s}$ whose length is bounded by the length of
$P'$. This yields a path from $y$
to $x$ between $a$ and $b$ of the appropriate length.
    \end{enumerate}
  \item\label{neuer_pkt} Exactly one vertex in $\{a,b\}$ lies in $A$.
Suppose that
$a$ lies in $A\setminus B$ and $b$ lies in $S_B\setminus A$.  In
particular, the vertex $a$ lies beneath $a_{l_s}$. Consider the
following mutually exclusive cases:
    \begin{enumerate}
 \item Both $x$ and $y$ lie in $S_B$. The direct path between them in
$S_B$ yields again the result.

 \item\label{neuer_punkt} Both $x$ and $y$ lie in $A$: If either  $x$
or $y$ equals
$a_{l_s}$, then one of them
lies over the other and the direct connection between them yields
the result. Otherwise, we may assume that both $x$ and $y$
lie beneath $a_{l_s}$. Let $Q$ be the path consisting of the
direct connection from $x$ to $a_{l_s}$ and from $a_{l_s}$ to $y$. If
the path $P$ connecting $x$ and $y$ necessarily passes through
$a_{l_s}$, then its length is at least the length of $Q$ and  the
result follows. Otherwise, since $A$ is simply connected, there is a
path connecting $x$ and $y$ of length at most $k$ between $a$ and
$a_{l_s}$, and thus, between $a$ and $b$.
    \item Exactly one, say $y$, is in $A$. Then $y$ must
lie beneath $x$ and the direct path between them yields the result.
    \end{enumerate}
    \end{enumerate}
\end{proof}

Since the only moment a vertex from $\A_{l_t}\cup\A_{r_t}$ was added
was in case $(\ref{neuer_pkt})(ii)$, namely $a_{l_s}$ (though only if
the original path passed through it), a careful analysis of the
previous proof yields the following, which corresponds to Axiom
($\Sigma4$) in \cite{BP00}; though we will not require its full
strength.

\begin{cor}\label{C:sc-oper-bp}
  A colored $N$-space $B$ with $\emptyset\leq B$ has the following
  property.  Given $t=[l_t,r_t]\subset [l,r]$, as well as $a$ in
  $\A_l(B)$ beneath $b$ in $\A_r(B)$, vertices $x$ and $y$ in
  $\A_t(B)$ lying between $a$ and $b$ and a path in $\A_t(B)$ of
  length $k$ connecting them, there is a path $P$ in $\A_t(B)$ between
  $a$ and $b$ connecting $x$ and $y$ of length at most $k$ such that
  all vertices in $P$ with levels $\A_{l_t}\cup\A_{r_t}$ come from the
  original path.
\end{cor}

By iterating Lemma \ref{L:sc-oper}, we obtain the following:
\begin{cor}\label{C:sc-oper}
  If $A$ is simply connected, then so is every strong extension of
$A$.
\end{cor}

The following observation can be easily shown.

\begin{lemma}\label{L:operation}
  Let $B$ be obtained from $A$ by applying the operation
  $\alpha_s$. Then, for every $t\subset\{0,\cdots,N\}$
    and every $x$ and $y$ in $\A_t(A)$,
  \[\dd^A_t(x,y)=\dd^B_t(x,y).\]
\end{lemma}

\begin{theorem}[Axioms]\label{T:axiome}
  Both Fra\"iss\'e limits $\mi$ and $\m_0$ have the following
elementary properties:
  \begin{enumerate}
  \item\label{S:AAA} simply connectedness.
  \item\label{S:axiome:unendlich} Given a finite subset $A$ and a
    non-empty interval $s=(l,r)$, for any two elements $a_l$ and $a_r$
    in $A$ with $a_r$ over $a_l$, there are paths
    \[a_l,b_{l+1},\ldots,b_{r-1},a_r\]
    \noindent such that the $s$-distance of $b_i$ to $\A_s(A)$ is
    arbitrarily large. In particular, if $s=\{i\}$, there is a new
    vertex $b_i$ not contained in $A$.
  \end{enumerate}
\end{theorem}

\begin{proof}

  \noindent(\ref{S:AAA}): This follows from Corollary
  \ref{C:sc-oper}.

  \noindent(\ref{S:axiome:unendlich}): After enlarging $A$, we may
  assume that $A\leq\mi$. One single application of $\alpha_s$ on
$(a_l,a_r)$ yields that $s$-distance of $b_i$ to $A$ is
infinite and remains so  at the end of the construction by
Lemma \ref{L:operation}.

  If we are considering $\m_0$, we may assume as well that $A\leq
  \m_0$. Furthermore, we may suppose that in order to build up $\m_0$
  from $A$, each of the operations $\alpha_i$, for $i$ in $s$, was
  applied $k$ many times consecutively on each of the new vertices in
  $\A_{i+1}$ and $\A_{i-1}$ between $a_l$ and $a_r$. Lemma
  \ref{L:operation} yields now the desired result.
\end{proof}

\begin{definition}
We will denote by $\psn$ the collection of sentences expressing
properties $(\ref{S:AAA})$ and $(\ref{S:axiome:unendlich})$ in
Theorem \ref{T:axiome}.
\end{definition}

\begin{definition}\label{D:complete}

  A \emph{flag} is a subgraph of a colored
$N$-space $M$ of the form $$a_0-\ldots-a_N,$$ \noindent where $a_i$
belongs to $\A_i(M)$ and they
form a path.

A set $D$ of a colored $N$-space $M$ is \emph{complete} if every
point in $D$ is contained in a flag in $D$.
\end{definition}

Observe that, if $D$ satisfies Axiom $(\ref{S:axiome:unendlich})$, it
is
complete.

\begin{definition}\label{D:nice}
 A subset $D$ of a colored $N$-space $M$ is \emph{nice} it satisfies
the following conditions:

\begin{enumerate}
\item For any two (possibly imaginary) points $a$ and $b$ in $D$,
\[ D_a^b= D\cap M_a^b.\]

 \item for all intervals
  $t\subset\{0,\dotsc,N\}$ and all $x$ and $y$ in $\A_t(D)$,

\[\dd_t^{M}(x,y)<\infty\;\Rightarrow\;\dd_t^{D}(x,y)<\infty.\]
\end{enumerate}

  A set $D$ is \emph{wunderbar} in $M$ if it satisfies the following:
\begin{enumerate}
\item For any two (possibly imaginary) points $a$ and $b$ in $D$,
\[ D_a^b= D\cap M_a^b.\]

 \item for all intervals
  $t\subset\{0,\dotsc,N\}$ and all $x$ and $y$ in $\A_t(D)$,

 \[\dd_t^{M}(x,y)=\dd_t^{D}(x,y).\]
\end{enumerate}

\end{definition}

 Clearly, wunderbar sets are nice.  As an application of the
operation $\alpha_s$ on $A$ does not yield connections between
the points of $A$ unless there was already one, the following result
follows immediately from Lemma \ref{L:operation}.
\begin{lemma}\label{L:opwunderbar}
  If $A\leq B$, then $A$ is wunderbar in $B$.
\end{lemma}

\begin{lemma}\label{L:stark_nice}
  Let $M$ be a simply connected colored $N$-space and $D$ nice
in $M$. Given an interval $s=[l,r]$ in $\{-1,\ldots,N+1\}$ and
$a_l\in\A_l(D)$ beneath $a_r\in\A_r(D)$, the set $D^{a_r}_{a_l}$ is
nice in $\A_s(M)$.
\end{lemma}

\begin{proof}
Since $D_a^b=D\cap M_a^b$ for any $a$ and $b$ in $D$, the
 first condition of niceness holds for $D^{a_r}_{a_l}$.

For the second condition, we may assume
that $a_l=-1$ by Remark \ref{R:lokal_operation}. Let $t\subset(-1,r]$
be an interval and vertices $x$ and $y$ in $\A_t(D)$ beneath $a_r$. We
need only show that, if $x$ and $y$ are connected in $\A_t(D)$,
then they are connected in $\A_t(D)$ beneath $a_r$. Let
$P$ be a path in $\A_t(D)$ connecting $x$ and $y$, but not
necessarily running beneath $a_r$. We call a vertex in $P$ avoidable if
it does not lie beneath $a_r$. Let $\A_n$ be the largest level
containing an avoidable vertex in $P$. Let $m$ be the number of
avoidable vertices in $P$ of level $n$. Choose $P$ such that the
pair $(n,m)$ is minimal for the lexicographical order. 

Given an avoidable vertex $b$ in $\A_n\cap P$, denote by $a'_1$ in
$\A_{l_1}$ the first non-avoidable vertex in $P$ between $b$ and $x$.
Likewise, let $a'_2$ in $\A_{l_2}$ be the first non-avoidable
vertex in $P$ between $b$ and $y$. Note that $l_1$ and $l_2$ are both
smaller than $n$, by maximality of $n$. Furthermore, since every 
avoidable direct neighbour of a non-avoidable vertex lies necessarily
in a larger level, by definition, it follows that both $l_1$ and $l_2$
are strictly smaller than $n$. Hence, the subpath $P'$ of $P$ between
$a'_1$ and $a'_2$ yields a connection in $\A_{t'}$, where
$t'=t\cap(-1,n]$ not passing through $a_r$. As $M$ is simply
  connected, there is a path $Q$ (with no repetitions) connecting
  $a'_1$ and $a'_2$ running beneath $a_r$. Now, the paths $Q$ and $P'$
  have only $a'_1$ and $a'_2$ as common vertices and they induce a
  loop. Remark \ref{R:simply}(\ref{R:kreis}) yields that $a'_1$ and
  $a'_2$ are $t_1$-connected, where $t_1=t\cap(-1,n)$. Since $D$ is
  nice, there is also a $t_1$-connection $R$ in $D$. Replacing $P'$ by
  $R$, we have a path whose avoidable vertices are still contained in
$(-1,n]$ and with fewer avoidable vertices of level $n$. Minimality of
$(n,m)$  shows that this path runs beneath $a_r$, as desired.
\end{proof}

\begin{cor}\label{C:AAAnice}
Let $D$ be nice in a colored $N$-space $M$. If $M$ is simply
connected, then so is $D$.
\end{cor}

\begin{lemma}\label{L:nice}
  Let  $A$ be a nice subset of a simply connected colored $N$-space
$M$. Consider a non-empty interval $s=(l,r)$
  and two vertices $a_{l_s}$ in $\A_{l_s}(A)$ and $a_{r_s}$ in
$\A_{r_s}(A)$ such that
  $a_{r_s}$ lies over $a_{l_s}$. Let $B \subset M$ be an extension of
$A$ given by new vertices
  $b_{l_s+1},\dotsc,b_{r_s-1}$ such that the sequence
 \[a_{l},b_{l+1},\dotsc,b_{r-1},a_{r}\] is a path. The
following are equivalent:
\renewcommand{\theenumi}{\alph{enumi}}
  \begin{enumerate}
  \item\label{L:nice:operation} The set $B$ is nice and obtained from
$A$ by applying
$\alpha_s$ on $(a_{l_s},a_{r_s})$.
  \item\label{L:nice:nice}  For some
(equivalently, all) $i$ in $s$, we have that
    $\dd_s^M(b_i,A)=\infty$.
 \item\label{L:nice:schwachnice} For some
(equivalently, all) $i$ in $s$, we have that
 $\dd^{M^{a_r}_{a_l}}(b_i,A)=\infty$.
  \end{enumerate}
\end{lemma}
Note that simply connectedness yields that
\[\dd^{M^{a_r}_{a_l}}(b_i,A)=\dd_{(l,r)}^M(b_i,A^{a_r}_{a_l}).\]
We say that $B$ is obtained from $A$ by a \emph{global application of}
$\alpha_s$ if it satisfies (any of) the above conditions.
In particular, the set $B$ is nice.

\begin{proof}
\noindent $(\ref{L:nice:operation})\to (\ref{L:nice:nice})$:
By the definition of $\alpha_s$ the distance $d_s^B(b_i,A)$ is
infinite
for every $i$ in $s$. Since  $B$ is nice in $M$, so is
$d_s^M(b_i,A)=\infty$.

\noindent $(\ref{L:nice:nice})\to (\ref{L:nice:schwachnice})$:
Obvious.

\noindent $(\ref{L:nice:schwachnice})\to (\ref{L:nice:nice}) $
If both $a_l$ and $a_r$ are imaginary, then there is nothing to
prove. Thus, may assume that $a_r$ is real. Furthermore,  suppose that
there is a  path $P$ connecting some $b_i$ with some $a$ in
$\A_s(A)$ in $\A_s(M)$. Take $P$ of shortest possible length.

We need to show that $$\dd^{M^{a_r}_{a_l}}(b_i,A)<\infty.$$
Note that $a$ and $a_r$ are connected in $\A_{(l,r]}(M)$ and, since
$A$ is nice, there is a shortest path $Q$ in $\A_{(l,r]}(A)$
witnessing this. In particular, let $a_{r-1}$ be the direct neighbour
of $a_r$ in $Q$. Connecting $Q$ and $P$, we have that $a_{r-1}$ and
$b_i$ lie beneath $a_r$ and are connected in $\A_{(l,r]}$ by a path
disjoint from $a_r$. Simply connectedness yields a path $Q_1$ beneath
$a_r$ in $\A_{(l,r)}$
connecting them. If $a_l$ is imaginary, we are done. Otherwise, the
vertices $a_{r-1}$ and $a_l$ are connected through $b_i$. Again
by simply connectedness, there is a  path $Q'$ connecting
them below $a_r$ in $[l,r)$. Let now $a_{l+1}$ be the direct neighbour
in $Q'$ above $a_l$
Note that $a_{l+1}$ and $b_i$ lie between $a_l$ and $a_r$.
Simply connectedness of  $M$ yields that there is a path in
$M^{a_r}_{a_l}$ between $b_i$ and $a_{l+1}$.
 Hence
$$\dd^{M^{a_r}_{a_l}}(b_i,A)<\infty.$$

\noindent $(\ref{L:nice:nice})\to (\ref{L:nice:operation})$:
If both $a_l$ and $a_r$ are imaginary, then  there are clearly no new
connections between any $b_i$ and $A$, and thus $B$ is obtained
by applying $\alpha_{[0,N]}$ to $A$. Hence, we may assume that $a_r$
is real.

We first need to show that no $b_i$ is in relation to an
element in $A$ besides $a_r$ and $a_l$. This implies that $B$ is
obtained from $A$ by application of
$\alpha_s$.
Assume first that $b_{r-1}$ is connected with some other element
$a_r'$ in $\A_r(A)$. Since $A$ is nice,
there is a path in $\A_{\{r-1,r\}}(A)$ connecting $a_r$ and $a_r'$.
This, together with the extra connection to $b_{r-1}$ yields a loop in
$\A_{\{r-1,r\}}$, which contradicts Remark \ref{R:simply}
(\ref{R:kreis}). Likewise for $b_{l+1}$. Finally, by assumption, no
$b_i$ in $\A_{(l+1,r-1)}$ is in relation with an element in $\A_s(A)$.

Now, in order to show that $B$ is nice, consider $x$ and $y$ in $B$
with finite $t$-distance in $M$. If both $x$ and $y$ lie in $A$, we
are done, since $A$ is nice. Likewise, if both $x$ and $y$ lie in the
path $a_{l},b_{l+1},\dotsc,b_{r-1},a_r$, the direct connection
works as well. Therefore, assume that $x$ lies in $A$ and $y$ does
not. By the assumption it follows that $t\nsubseteq s$. Suppose
that $l$ lies in $t$. Since $y$ and $a_l$ are $t$-connected (in $M$),
so are $x$ and $a_l$. As $A$ is nice, there is a connection between
$x$ and $a_l$ in $\A_t(A)$. In particular, there is a connection
between $x$ and $y$ in  $\A_t(B)$.

\end{proof}

\begin{theorem}\label{T:extnice}
 Let $M$ be complete and simply connected. Given a nice subset $A$ and
$b$ in $M$, there is a nice subset $B$ of $M$ containing $b$ such that
$A\leq B$ in finitely many steps.
\end{theorem}

\begin{proof}
We may clearly assume that $b$ does not lie in $A$.

Let $r$ be minimal such that there exists an element $a_r$ in
$\A_r(A)$ lying over $b$ (if $r=N+1$, set $a_r=a_{N+1}$). Likewise,
choose $l$ maximal such that there exists an element $a_l$ in
$\A_l(A)$ beneath $b$ (if $l=-1$, then set $a_l=a_{-1}$). We call the
interval $s=(l,r)$ the width of $b$ over $A$. Define as well the
\emph{distance} from $b$ to $A$ as $$\dd_s(b,A^{a_r}_{a_l}).$$

We prove the theorem by induction on the width and the distance from
$b$ to $A$: If the distance is infinite, by completeness of $M$,
choose a path \[a_l,b_{l+1},\dotsc,b_{r-1},a_r,\]

\noindent passing through $b$. By Lemma \ref{L:nice}, the
set $A\cup\{b_{l+1},\dotsc,b_{r-1}\}$ obtained from $A$ by applying
$\alpha_s$ is nice and contains $b$.

Otherwise, let $P$ be a path of minimal length lying between
$a_l$ and $a_r$ connecting $b$ to $A$. Let $b'$ be the last element
in $P$ before $b$. By assumption, the distance from $b'$ to $A$ is
strictly smaller than the length of $P$. Thus, there is a nice set
$B'\geq A$ containing $b'$. Either the width or the distance of $b$ to
$B'$ has become smaller and we can now finish by induction.

\end{proof}

In particular, we can now prove that the notions of nice and
wunderbar agree.

\begin{cor}\label{C:nice=wund}
A nice subset $A$ of a complete simply connected set $M$ is
wunderbar.
\end{cor}
\begin{proof}
 Suppose we are given two points $a$ and $b$ in $A$ and an $s$-path
$P$ in $M$ of length $n$ connecting them. By Theorem \ref{T:extnice},
we
can obtain a nice set $B$ such that $A\leq B$ and $B$ contains the
path $P$. By Lemma \ref{L:opwunderbar}, the set $A$ is wunderbar in
$B$, so there is an $s$-path of length $n$ in $A$ connecting $a$ and
$b$. Thus, the set $A$ is wunderbar.

\end{proof}

Combining the previous results, we obtain the following.

\begin{cor}\label{C:nice_modell}
Let $M$ be complete and simply connected and
$A$ be a nice subset. The following hold:
\renewcommand{\theenumi}{\alph{enumi}}
  \begin{enumerate}
 \item\label{F:nice_modell:op} If $M\setminus A$ is countable,
then $A\leq M$.
  \item\label{F:nice_modell:modell} $A$ is simply connected.
 \item\label{F:nice_modell:very_nice} $A$ is wunderbar.
  \item\label{F:nice_modell:nice_op} If $A$ is countable, then
$\emptyset\leq A$.
  \end{enumerate}
\end{cor}
\begin{proof}
Theorem \ref{T:extnice} yields $(\ref{F:nice_modell:op})$. Now,
Corollary \ref{C:AAAnice} yields $(\ref{F:nice_modell:modell})$. In
order to prove $(\ref{F:nice_modell:very_nice})$, it is sufficient to
consider countable nice subsets $A$.
Replace $M$ by
a countable elementary substructure $M'$ that contains  $A$. Then $A$ is nice
in $M'$ and $A\leq M'$ by \ref{F:nice_modell:op}.
Lemma \ref{L:opwunderbar} yields that $A$ is
wunderbar in $M'$ and hence in $M$. Since $\emptyset$ is nice,
clearly $(\ref{F:nice_modell:nice_op})$ follows from
$(\ref{F:nice_modell:op})$ and $(\ref{F:nice_modell:modell})$.
\end{proof}

It follows that, for countable $A$, we have $\emptyset\leq A$ if and
only if $A$ is simply connected and complete. And for simply connected
complete countable $B$, we have that $A\leq B$ \iff $A$ is nice in
$B$. Therefore

\begin{cor}
  The model $\mi$ is the Fra\"iss\'e limit of the class of finite
complete simply connected colored $N$-spaces together with nice
embeddings.
\end{cor}

\noindent The construction is actually simpler than the general
construction given in Section~\ref{S:Fraisse}, since if a finite
set $B$ satisfies that $B_a^b=B\cap M_a^b$ for all $a$ and $b$ in $B$,
then $B$ is r-strong in $\mi$ if and only it is nice in $\mi$. Indeed,
consider a rich sequence $A_0\leq A_1\leq\ldots$ with union $\mi$.
Then $B$ is contained some $A_i$. But $B$ is also nice in $A_i$, which
implies $B\leq A_i$, and therefore $B$ is r-strong in $\mi$.

Having $\mi$ as a model, the theory $\psn$ is consistent. It
will follow from the next proposition that it is complete. In
particular, the stronger version of Axiom $(\ref{S:AAA})$ stated in
Corollary \ref{C:sc-oper-bp} follows formally from our axioms.

\begin{prop}\label{L:hinher}
Any two $\omega$-saturated models of $\psn$ have the back-and-forth
property with respect to partial isomorphisms between finite nice
substructures.
\end{prop}

\begin{proof}
  Let $M$ and $M'$ be two $\omega$-saturated models and consider a
  partial isomorphism $f:A\to A'$, where $A$ is nice in $M$ and $A'$
  is nice in $M'$.

  Given $b$ in $M$, Theorem \ref{T:extnice} yields a nice finite
  subset $B\geq A$ containing it. Thus, we may assume that $B$ is
  obtained from $A$ by applying $\alpha_s$ on  $(a_l,a_r)$.  Since
$M'$ is an $\omega$-saturated model of Axiom
  $(\ref{S:axiome:unendlich})$, there is a path
  $a'_l,b'_{l+1},\ldots,b'_{r-1},a'_r$ in $M'$ such that the
  $s$-distance of $b'_i$ to $A'$ is infinite. By Lemma \ref{L:nice})
  the set $B'=A'\cup\{b'_{l+1},\ldots,b'_{r-1}\}$ is nice and $f$
  extends to an isomorphism between $B$ and $B'$.
\end{proof}

\begin{theorem}\label{T:el}
  Any partial isomorphism $f:A\to A'$ between two finite nice subsets
  of two models of $\psn$ is elementary.
\end{theorem}

\begin{proof}
  Replace the models $M$ and $M'$ by two $\omega$-saturated extensions
  $M_1$ and $M'_1$ Note that $A$ and $A'$ remain nice in the
  corresponding extensions. Lemma \ref{L:hinher} yields that $f$ is
  elementary with respect to $M_1$ and $M'_1$ and thus its restriction
  to $M$ and $M'$ is elementary as well.
\end{proof}
\begin{cor}
  The theory $\psn$ is complete.
\end{cor}
\begin{proof}
 Note that set $\emptyset$ is nice in any colored $N$-space and apply
Theorem \ref{T:el}.
\end{proof}

\begin{cor}\label{C:complete}
  The type of a nice set $A$ is determined by its quantifier-free
type. 
\end{cor}

\begin{cor}\label{C:omegasat}
  The model $\mi$ is $\omega$-saturated.
\end{cor}

\begin{proof}
  Let $M$ be any $\omega$-saturated model of $\psn$.  It follows from
  Lemma \ref{L:fraisse} and the equality of nice and r-strong that the
  family of isomorphisms between finite nice subset of $M$ and $\mi$
  has the back-and-forth property. This implies that $\mi$ is also
  $\omega$-saturated.
\end{proof}

\begin{cor}\label{C:prim}
  The Fra\"iss\'e limit $M_0$ is the prime model of $\psn$.
\end{cor}
\begin{proof}
  Consider any finite $A\subset M$ which can be obtained from some
fixed flag by a sequence of applications of
$\alpha_{\{i\}}$ for varying $i\in[0,N]$.  Since the
$\dd_{\{i\}}$-distances are either $0$ or $\infty$, it follows
inductively from Lemma \ref{L:nice} that all  intermediate sets are
nice.  So the quantifier-free type of $A$ implies that $A$ is nice and
therefore implies the type of $A$. Whence $A$ is atomic. This shows
that $\m_0$ is atomic.
\end{proof}

\begin{cor}\label{C:Niceacl}
Nice sets are algebraically closed. 
\end{cor}

\begin{proof}
 By Corollary \ref{C:omegasat}, we may assume that the 
nice set $A$ is a subset of $\mi$. By Corollary
\ref{C:nice_modell} $(\ref{F:nice_modell:op})$, we have that $M$ is
an increasing union of nice sets containing $A$. Thus, we may reduce
the statement to showing that if
$B=A\cup\{b_{l_s+1},\dotsc,b_{r_s-1}\}$ is obtained by applying the
operation $\alpha_s$ on $a_{l_s},a_{r_s}$ in $A$, then the tuple
$(b_{l_s+1},\dotsc,b_{r_s-1})$ has infinitely many $A$-conjugates.
This is now clear, as any two sets resulting from applying the
operation $\alpha_s$ on $a_{l_s},a_{r_s}$ in $A$ have the same type
over $A$, by Lemma \ref{L:nice} and Corollary \ref{C:complete}.
 
\end{proof}


\section{Words and letters}\label{S:Woerte}

In this section, we will study the semigroup $\cox$ generated by the
operations $\alpha_s$, where $s$ stands for a non-empty interval in
$[0,N]$. Such intervals will be then called \emph{letters}. We will
exhibit a normal reduced form for words in $\cox$ and describe the
possible interactions between words when multiplying them.

Two letters $s$ and $t$ in $[0,N]$ \emph{commute} if their distance
is at least $2$. That is, either $r_s\leq l_t$
or $r_t\leq l_s$, where $s=(l_s,r_s)$ and $t=(l_t,r_t)$. By
definition, no letter commutes with itself nor with any
proper subletter.

\begin{definition}\label{D:Cox}
  We define $\cox$ to be the monoid generated by all letters
  in $[0,N]$ modulo the following relations:
  \begin{itemize}
  \item $ts=st=s$ if $t\subset s$,
  \item  $ts=st$ if $s$ and $t$ commute.
  \end{itemize}
  We denote by $1$ the empty word.
\end{definition}

The \emph{inversion} $u\mapsto u\inv$ of words defines an
antiautomorphism of $\cox$. All concepts introduced from now on
will be invariant under inversion.

The  \emph{centraliser} $\Cent(u)$  of a  word $u$ in $\cox$ is the
collection of all indexes in $[0,N]$ commuting with every letter
in $u$. Clearly, a letter $s$ commutes with $u$ in $\cox$ if and only
if  $s\subset C(u)$.

In order to obtain a normal form for elements in $\cox$, we say that a
\emph{word} $s_1\cdots s_n$ is \emph{reduced} if there
is no pair $i\not=j$ of indices such that $s_i\subset s_j$ and $s_i$
commutes with all $s_k$ with $k$ between $i$ and $j$.

\begin{definition}\label{D:reduction}
  The word $u$ can be \emph{reduced to} $v$, denoted by $u\to v$, if
  $v$ is obtained from $u$ by finitely many iterations of the
  following rules:
  \begin{description}
  \item[\sc Commutation] Replace an occurrence of $s\cdot t$ by
    $t\cdot s$, if $s$ and $t$ commute.
  \item[\sc Cancellation] Replace an occurrence of $s\cdot t$ or
    $t\cdot s$ by $s$, if $t\subset s$.
  \end{description}
  Two words $u$ and $v$ are \emph{equivalent} (or $u$ is a
\emph{permutation} of $v$), denoted by $u\approx v$, if
  $u\to v$  by  exclusively applying the commutation rule.
\end{definition}

It is easy to see that permutations of reduced words remain
reduced. In particular, a word is reduced if and only if the
cancellation rule cannot be applied to any permutation.

Clearly, two word $u$ and $v$ represent the same element in $\cox$
if $u\to v$.  The following proposition yields in particular that the
converse is  true: Two words have a common reduction if they represent
the same element in $\cox$ (\emph{cf.} Corollary
\ref{C:cox_reduced}).

\begin{prop}\label{P:nonsplitred}
  Every word $u$ can be reduced to a unique (up to equivalence)
  reduced word $v$. We refer to $v$ as the \emph{reduct of} $u$.
\end{prop}

\begin{proof}
  Among all possible reductions of the word $u$, choose $v$ of minimal
length. Clearly, cancellation cannot be applied any further to
a permutation of $v$, thus $v$ is reduced. We need only show that $v$
is unique such.

For that, we first introduce the following rule:

  \begin{description}
    \item[\sc Generalised Cancellation] Given a word $s_1\cdots
s_n$ and a pair  of indices $i\not=j$ such that $s_i\subset s_j$ and
$s_i$ commutes with all $s_k$'s with $k$ between $i$ and $j$, then
delete the letter $s_i$.
  \end{description}
  If the situation described above occurs, we say that
  $s_i$~\emph{is absorbed by} $s_j$. Note that a generalised
cancellation is
obtained by  successive commutations and one single cancellation.
Furthermore, one single cancellation applied to some permutation of $u$
can be obtained as some permutation of a generalised cancellation
applied to $u$. This implies that every reduct can be obtained by
a sequence of generalised cancellations followed by a permutation.

  Assume now that $u\to v_1$ and $u\to v_2$,  where both $v_1$ and
$v_2$ are reduced. We will show, by induction on the length of
$u$, that $v_2$ is a permutation of $v_1$. If $u$ is itself reduced,
then $v_1$ and $v_2$ are permutations of $u$ and hence the result
follows. Otherwise, there are two words $u_1$ and $u_2$ obtained
from $u$ by one single generalised cancellation such that
$u_i\to v_i$ for $i=1,2$.

  We claim that there is a word $u'$ such that $u_i\to u'$ for
  $i=1,2$, either by permutation or by a single generalised
  cancellation. This is immediate except for the case where there are
  indices $i$, $j$ and $k$ (for $i\neq k$) such that $u_1$ is obtained
  from $u$ because the letter $s_i$ is absorbed by $s_j$ and $u_2$ is
  obtained from $u$ in in which the same letter $s_j$ is absorbed by
  $s_k$. In this case, set $u'$ to be the word obtained from $u$ by
  having both $s_i$ and $s_j$ absorbed by $s_k$.  Clearly, we have
  that $u_1\to u'$. Also, since $s_i\subset s_j$, it follows that
$s_i$  commutes also with all letters between $s_j$ and $s_k$. Hence,
the word $u'$  is obtained from $u_2$ in which $s_k$ absorbs $s_i$.
Let $v'$ be a reduct of $u'$. Induction applied to $u_1$ and $u_2$
implies that  $v'$ is a permutation of both $v_1$ and $v_2$. Hence,
the word $v_1$  is a permutation of $v_2$.
\end{proof}

\begin{cor}\label{C:cox_reduced}
  Every element of $\cox$ is represented by a reduced word, which is
  unique up to equivalence.
\end{cor}
\begin{proof}
  Let $C$ be the collection of equivalence classes of reduced
words. From the previous result, it follows that there is a natural
surjection $C\to\cox$. Represent by $[u]$ the equivalence class of
the word $u$. Set

  \[[u] \cdot [v] = [w] \quad\text{ iff }\quad
u\cdot v\to w
  .\]
  Then $C$ has a natural semigroup structure. Since $C$ satisfies
the defining relations of $\cox$, the map $C\to\cox$ is an isomorphism.
\end{proof}

In order to exhibit a canonical representative of the equivalence
class $[u]$, we introduce the following partial ordering on
letters:
\[(l_s,r_s)<(l_t,r_t)\;\;\text{iff}\;\; r_s\leq l_t.\]
A reduced word $s_1\cdots s_n$ is in \emph{normal
  form} if for all $i<n$, if $s_i$ and $s_{i+1}$ commute, then
$s_i<s_{i+1}$.
\begin{remark}
  Every reduced word is equivalent to a unique word in normal form.
\end{remark}
\begin{proof}
 We will actually prove a more general result: Let $S$ be any set
equipped with a partial order $<$. We say that $s$ and $t$ commute if
either $s<t$ or $t< s$. Let $S^\ast$ be the semigroup
generated by $S$ modulo commutation. Two words in $S^\ast$ are
equivalent if they can be transformed into each other by successive
commutations of  adjacent elements. A word $s_1\cdots s_n$ is in
normal form if $s_i\not>s_{i+1}$ for all $i<n$. We have the following.

\begin{claim}
  Every word $u$ in $S^\ast$ is equivalent to a unique word
  $v$ in normal form.
\end{claim}

  For existence, start with $u$ and swap successively every pair
  $s_{i}>s_{i+1}$. This process must stop since the number of
inversions $\{(i,j)\mid i<j\text{ and }s_i>s_j\}$ is
  decreased by $1$ at every step. The resulting $v$ is in normal
form.

  For uniqueness, consider two equivalent words in normal
form $u=s_1\cdots s_n$ and $v=t_1\cdots t_n$ . Let $\pi$ be some
permutation transforming $u$ into $v$. Suppose for a contradiction
that $\pi(1)=k\neq 1$. Then $t_k=s_1$ commutes with  $t_i$ for
$i<k$. By hypothesis, we have $t_{k-1}<t_k$. Note that there is no
$i<k$ with $t_i<t_k$ and $t_k<t_{i-1}$. Hence, for all $i<k$, we have
that $t_i<t_k$ and thus $t_1<t_k$, that is, $t_1<s_1$. By means of the
permutation $\pi\inv$, we conclude that $s_1<t_1$, which yields a
contradiction. Thus $\pi(1)=1$ and hence $s_2\cdots  s_n$ is
equivalent to $t_2\cdots t_n$. Induction on $n$
yields the desired result.
\end{proof}

\noindent It is an easy exercise to show that, for $S$ and $S^\ast$ as
before, we have  \[r\cdot t_2\cdots t_n\approx
r\cdot s_2\cdots s_n\;\Rightarrow\;t_2\cdots t_n\approx
s_2\cdots s_n.\] Therefore, we obtain the following result.
\begin{remark}
  $u\cdot v\approx u\cdot v'$ implies $v\approx v'$.
\end{remark}

Given two reduced words $u=s_1\cdots s_m$ and $v=t_1\cdots
t_n$, their product $u\cdot v$ is not reduced \iff one of the two
following cases occurs:
\begin{itemize}
\item There are $i\leq m$ and $j\leq n$ such that $s_i$ commutes with
  $s_{i+1}\cdots s_m$ and with $t_1\cdots t_{j-1}$ and it is
contained in $t_j$.
\item There are $j\leq n$ and $i\leq m$ such that $t_j$ commutes with
  $t_1\cdots t_{j-1}$ and with $s_{i+1}\cdots s_m$ and it is
contained in $s_i$.
\end{itemize}
Based on the previous observation, we introduce the following
definition.
\begin{definition}\label{D:Endstueck}
Given two words $u=s_1\cdots s_m$ and $v=t_1\cdots t_n$
words, we say that:
  \begin{enumerate}
  \item $s_i$ belongs to the \emph{final segment} of $u$ if $s_i$
    commutes with $s_{i+1}\cdots s_m$.
  \item The letter $s$ is (properly) \emph{left-absorbed} by $v$ if
 it commutes with with $t_1\cdots t_{j-1}$ and is
a (proper) subset of $t_j$ for some $j\leq n$. A word is (properly)
left-absorbed by
    $v$ if all its letters are (properly) left absorbed by $v$.
  \item $v$ \emph{bites $u$ from the right} if $v$ left-absorbs some
    element in the final segment of $u$.
  \end{enumerate}
  The concepts \emph{initial segment}, \emph{right-absorbed} and
  \emph{left-biting} are defined likewise.
\end{definition}
Clearly, these notions depend only on the equivalence
class of $u$ and $v$. Thus, the following lemma follows.
\begin{lemma}
  Given two reduced words $u$ and $v$, the product $u\cdot v$ is
  reduced \iff none of them bites the other one (in the corresponding
  directions).
\end{lemma}

If both $u$ and $v$ are reduced and $u$ is absorbed by $v$, then
$u\cdot v$ reduces to $v$. Corollary
\ref{C:absorb_cox} will show that the converse also holds.

The following observations will be often used throughout this
article.
\begin{lemma}[Absorption Lemma]\label{L:schlucktheorie}
  Let $v$ be a (possibly non-reduced) word.
  \begin{enumerate}
  \item\label{L:schlucktheorie:1} If a letter $s$ is left-absorbed by
    $v$, then there is a unique letter in $v$ witnessing it.
  \item\label{L:schlucktheorie:2} If two non-commuting letters are
    absorbed by $v$, then they are absorbed by the same letter in $v$.
  \item\label{L:schlucktheorie:3} Suppose $v=v_1\cdot
    v_2$ and let $u$ be a word left-absorbed by $v$ but not bitten
from the right by $v_1$, then $u$ and $v_1$ commute and $u$ is
left  absorbed by $v_2$.
  \end{enumerate}
\end{lemma}
\begin{proof}
  Assume $v=t_1\cdots t_n$. Let $r\subset t_i$ commute with $t_1\cdots
  t_{i-1}$ and $s\subset t_j$ commute with $t_1\cdots t_{j-1}$. Assume
  $i\leq j$. Then, either $i=j$ or $s$ commutes with $t_i$, which
  implies that $s$ commutes with $r$. This yields both
  (\ref{L:schlucktheorie:1}) and (\ref{L:schlucktheorie:2}).

  For (\ref{L:schlucktheorie:3}), we apply induction on the
length $m$ of $u=s_1\cdots s_m$. If $m=0$, then there is nothing to
prove. Otherwise, the subword $u'=s_2\cdots s_m$ is not bitten by
$v_1$ by assumption. Induction gives that $u'$ commutes
with $v_1$ and is absorbed by $v_2$. The letter $s_1$
  cannot be absorbed by $v_1$, for otherwise $s_1$
  would also commute with $u'$ and thus it would belong to the final
segment of $u$. The word $u$ would then be bitten
by $v_1$. Since $s_1$ is  absorbed by $v$ but not by $v_1$, it must
commute with $v_1$ and hence it is absorbed by $v_2$ as well.
\end{proof}

Based on the the previous result, we introduce the following notions.
\begin{definition}\label{D:Schluckindizes}
  The \emph{left stabiliser} $\sL(v)$ of a word $v=t_1\cdots t_n$ is the
  union of the sets
  \[\sL^j(v)=  t_j \cap \Cent(t_1\ldots t_{j-1}).\] The
  \emph{right stabiliser} $\sr(v)$ is defined likewise or,
  alternatively, as $\sL(v\inv)$
\end{definition}

By Lemma \ref{L:schlucktheorie}(\ref{L:schlucktheorie:2}), the sets
$\sL^j(v)$
are either empty or intervals commuting with each other. Equivalent
words have same stabilisers. In fact, if $u\to v$ then
$\sL(u)\subset \sL(v)$.

\begin{lemma}\label{L:v_in_sl}
 The letter $s$ is absorbed by $v$ \iff $s\subset \sL(v)$.
\end{lemma}

Set
\[|s_1 \dotsb s_m| = s_1\cup\dotsb\cup s_m.\]
Then $u$ is absorbed by $v$ \iff $|u|\subset \sL(v)$.  Furthermore,
the word $v$ bites $u$ from the right \iff some element in the final
segment of $u$ is contained in $\sL(v)$.

\begin{lemma}\label{L:schluckteil}
  Given two words $u$ and $v$, there is a
  unique decomposition $u=u_1\cdot u_2$ (up to commutation) such that:
  \begin{itemize}
  \item $u_2$ is left-absorbed by $v$.
  \item $u_1$ is not bitten from the right by $v$.
  \end{itemize}
  The decomposition of $u$ depends only on the set $\sL(v)$.
\end{lemma}

\begin{proof}
  We proceed by induction on the length of $u$. If $u$ is not bitten
by $v$, we set $u_1=u$ and $u_2=1$. Otherwise, up to permutation,
we have  $u=u'\cdot s$, where $s$ is absorbed by $v$. Decompose
$u'$ as $u'_1\cdot u'_2$ and set  $u_1=u'_1$ and $u_2=u'_2\cdot
s$.

Uniqueness is proved in a similar fashion.
\end{proof}

We can now  describe the general form of the product of two
reduced words in $\cox$.

\begin{theorem}[Decomposition Lemma]\label{T:einfach_amador}
  Given two reduced words $u$ and $v$, there are unique
  decompositions (up to permutation):
  \begin{align*}
    u&=u_1\cdot u'&
    v'\cdot v_1&=v,
  \end{align*}
  such that:
  \renewcommand{\theenumi}{\alph{enumi}}
  \begin{enumerate}
  \item $u'$ is left-absorbed by $v_1$,
  \item $v'$ is properly right-absorbed by $u_1$,
  \item $u'$ and $v'$ commute,
  \item $u_1\cdot v_1$ is reduced.
  \end{enumerate}
\end{theorem}
\noindent It follows that $u\cdot v \to u_1\cdot v_1$. We call such a
decomposition \emph{fine}.
\begin{proof}
  We apply Lemma \ref{L:schluckteil} to $u$ and $v$
  to obtain a decomposition
  \begin{equation*}
    u=u_1\cdot u',
  \end{equation*}
  such that $u'$ is left-absorbed by $v$ and
  $u_1$ is not bitten by $v$ from the right.
  The same (in the other direction) with $u_1$ and $v$ yields
\begin{equation*}
    v'\cdot v_1=v,
  \end{equation*}
  where $v'$ is right-absorbed by $u_1$ and $v_1$ is not bitten
  from the left by $u_1$.

  First, we show $(c)$, that is, the words $u'$ and $v'$
commute. If
not, let $s$ the first
  element of $u'$ which does not commute with $v'$.  Since $s$ is
  left-absorbed by $v'\cdot v_1$, it must be left-absorbed by $v'$.
  As $u_1$ right-absorbs $v'$, it also right-absorbs $s$, which
  contradicts that $u_1\cdot u'$ is reduced.
  Lemma \ref{L:schlucktheorie}(\ref{L:schlucktheorie:3}) gives that
$u'$ is absorbed by
  $v_1$, showing $(a)$.

  Let us now show $(d)$: the product $u_1\cdot v_1$ is reduced.
Otherwise,  as $v_1$ is not bitten from the left by $u_1$,
it bites $u_1$ from the right, i.e.\ it left-absorbs a letter $s$ from
the final segment of $u_1$. The Absorption Lemma
\ref{L:schlucktheorie}, applied to $u_1=u_1^1\cdot s$ and $v'$, which
is right absorbed by $u_1$, gives (possibly after permutation) a
decomposition $v'=x\cdot y$, where $|x|\subset s$ and $y$ commutes
with $s$. There are two cases:
  \begin{enumerate}
  \item The word $x=1$. Then $s$ commutes with $v'$ and is absorbed
by $v_1$. This contradicts that $u_1$ is not bitten by $v_1$ from the
right.
  \item The word $x$ is not trivial. As it is absorbed by
$s$ and $s$ is right-absorbed by $v_1$, we have that $x$ is
right-absorbed by $v_1$.  This contradicts that $v'\cdot v_1$ is
reduced.
  \end{enumerate}

  The only  point left to prove is that $v'$ is properly
right-absorbed by $u_1$. Otherwise, there is a letter $t$ in
$v'$ which is absorbed but not properly absorbed by $u_1$. Then $t$
occurs in the final segment of $u_1$ and $v'= t\cdot y$ up to
commutation. In particular, the word $u_1$ is bitten from the right by
$v'$ and thus by $v$, which contradicts our choice of $u_1$.

  In order to show uniqueness, assume we are given another fine
decomposition:
  \begin{align*}
    u&=u_1\cdot u'&
    v'\cdot v_1&=v
  \end{align*}
We need only show the following four facts:
\begin{enumerate}
  \item The word $u'$ is left-absorbed by $v$: Since $u'$ commutes
with $v'$ and is left-absorbed by $v_1$, then it is left-absorbed
by $v'\cdot v_1$ as well.
  \item The word $u_1$ is not bitten by $v$ from the right: Suppose
not and take   a letter $s$ in the final segment of $u_1$ which is
left-absorbed by $v$. Since $u_1\cdot v_1$ is reduced, the letter $s$
must be left-absorbed by $v'$. Let $t$ in $v'$ containing $s$.
However, the word $t$ is right-absorbed by $u_1$. As $u_1$ is
reduced and $s$ is in the final segment of $u_1$, the only
possibility is that $s=t$. But then $t$ is not properly left-absorbed
by $u_1$, which is a contradiction.
  \item $v'$ is right-absorbed by $u_1$: By definition.
  \item $v_1$ is not bitten from the left by $u_1$: This
clearly follows from the fact that $u_1\cdot  v_1$ is reduced.
  \end{enumerate}
\end{proof}

\begin{cor}\label{C:absorb_cox}
  Let $u$ and $v$ be reduced words. Then $v$ left-absorbs $u$ if and
  only if $uv=v$ in $\cox$.
\end{cor}
\noindent Note that $uv=v$ in $\cox$ if and only if $u\cdot v\to v$.
\begin{proof}
 Clearly, if $v$ left-absorbs $u$, then $u\cdot v\to v$. For the
converse, apply the Decomposition Lemma \ref{T:einfach_amador} to
$u$ and $v$ to obtain:
  \begin{align*}
    u&=u_1\cdot u'&
    v'\cdot v_1&=v
  \end{align*}
  such that $u'$ is left-absorbed by $v_1$, the word $v'$ is properly
  right-absorbed by $u_1$, the words $u'$ and $v'$ commute and
$u_1\cdot v_1$ is reduced. By assumption,we have
$$u\cdot v\to u_1\cdot v_1\approx v=v'\cdot v_1.$$
 Thus $u_1=v'$. Since $u_1$ must be properly right-absorb itself,
this forces $u_1$ to be trivial. Hence $u=u'$ is
  left-absorbed by $v$.
\end{proof}
As in $\cox$ (or generally, in any semi-group), the identity
$uvx=uv$ holds if $vx=v$, we have the following.
\begin{cor}\label{C:sr_monoton}
  Let $u$ and $v$ be reduced words and $w$ the reduct of $u\cdot
  v$. Then $\sr(v)\subset\sr(w)$.
\end{cor}

\begin{definition}\label{D:Wob}
  The \emph{wobbling} between two words is
  \[\wob(u,v)=\sr(u)\cap\sL(v).\]

\end{definition}

\begin{remark}\label{R:wob_echt}
  If $u\cdot v$ is reduced, then every $s\subset\wob(u,v)$ is properly
  right-absorbed by $u$ and properly left-absorbed by $v$.
\end{remark}
\begin{proof}
  If $s$ is not properly right-absorbed by $u$, then $s$ belongs to the
  final segment of $u$. Since $s$ is left-absorbed by $v$, the
product $u\cdot v$ would not be reduced.
\end{proof}
\begin{lemma}\label{L:Wobschluck}
  Assume that $v_1\cdot v_2$ and $u\cdot v_2$ are reduced. If $v_1$ is
  right absorbed by $u$, then \[\wob(v_1\cdot v_2, h)\subset
  \wob(u\cdot v_2,h).\]
\end{lemma}
\begin{proof}
  The word $u\cdot v_2$ is the reduct of $u\cdot (v_1\cdot v_2)$.
Corollary  \ref{C:sr_monoton} yields that  $\sr(v_1\cdot
v_2)\subset \sr(u\cdot
v_2)$.
\end{proof}
We will now study the  idempotents of $\cox$.

\begin{definition}
  A word is commuting if it consists of pairwise commuting letters.

 The letters of the final segment of a word $u$ form a commuting
word, which we denote by $\tilde u$ (up to equivalence).
\end{definition}

 Commuting words are automatically reduced.  Since every subset of
$[0,N]$ can uniquely be written as the union of commuting intervals, a
commuting word (up to equivalence) can be considered as just a set of
numbers.  The following is an easy observation:

\begin{lemma}\label{L:segment_zerlegung}
  Every word $u$ is equivalent to a word $x\cdot\tilde u$, where
  $\tilde u$ is the final segment of $u$.
\end{lemma}

\noindent Note that no letter in the final segment of $x$ commutes
with $\tilde u$.

\begin{prop}\label{P:proper_absorbed}
  Let $u$ and $v$ reduced words such that $v$ left-absorbs $u$. Then,
  up to permutation, there is are unique decompositions
  \begin{align*}
    u&=u'\cdot w&
    w\cdot v'&=v,
  \end{align*}
  such that
  \begin{enumerate}
  \item $u'$ is properly left-absorbed by $v'$,
  \item $w$ commutes with $u'$,
  \item $w$ is a commuting word.
  \end{enumerate}
\end{prop}
\begin{proof}
  Apply the Absorption Lemma \ref{L:schlucktheorie} to $v$ and
$u$, which is completely left-absorbed by $v$. The letters of $u$
which are not properly left-absorbed by $v$ must commute with all
other letters and form the word $w$.
\end{proof}
We obtain therfore the following consequence, which implies that a
word is commuting if and only if it is an idempotents in
$\cox$.
\begin{cor}\label{C:kommwort}
  A reduced word is commuting if and only if it absorbs itself (left,
  or equivalently, right).
\end{cor}

\begin{proof}
  Clearly, if $u$ is commuting, then $|u|=\sL(u)$, so $u$ absorbs
  itself. Suppose now that $u$ left-absorbs itself. By the proposition
  applied to $v=u$ we find $u=w\cdot u'\approx w\cdot v'$ such that
  $u'$ is properly left-absorbed by $v'$ and $w$ is a commuting
  word. It follows that $u'=v'$ properly absorbs itself, i.e. the word
$u'=1$.
\end{proof}
We can now state a symmetric version of the Decomposition Theorem
\ref{T:einfach_amador}, combined  with
Proposition \ref{P:proper_absorbed}.

\begin{cor}[Symmetric Decomposition Lemma]\label{C:amador}
 Let $u$ and $v$ be two reduced words. Each can be uniquely decomposed (up to
 commutation) as:
  \begin{align*}
    u&=u_1\cdot u'\cdot w&
    w\cdot v'\cdot v_1&=v,
  \end{align*}
such that:
\renewcommand{\theenumi}{\alph{enumi}}
  \begin{enumerate}
  \item $u'$ is properly left-absorbed by $v_1$,
  \item $v'$ is properly right-absorbed by $u_1$,
  \item $u'$, $w$ and $v'$ pairwise commute,
  \item $w$ is a commuting word,
  \item $u_1\cdot w\cdot v_1$ is reduced.
  \end{enumerate}
  In particular, we have $u\cdot v \to u_1\cdot w\cdot v_1$.
\end{cor}

\begin{figure}[!htbp]
\centering
\begin{tikzpicture}[>=latex,text height=.1ex,text depth=0.1ex]

\fill (0,0) circle (2pt);
\fill (1,1) circle (2pt);
\fill (-1,1) circle (2pt);
\fill (2,2) circle (2pt);
\fill (-2,2) circle (2pt);
\fill (1,3) circle (2pt);
\fill (-1,3) circle (2pt);
\fill (3,3.5) circle (2pt);
\fill (-3,3.5) circle (2pt);

\draw[->] (0,0) -- (1,1) node[pos=.5, below right] {$w$};
\draw[->] (1,1) -- (2,2) node[pos=.5, below right] {$v'$};
\draw[->] (2,2)  --(3,3.5) node[pos=.5, below right] {$v_1$} ;
\draw[->] (-3,3.5) -- (-2,2) node[pos=.5, below left] {$u_1$};
\draw[->] (-2,2)  --(-1,1) node[pos=.5, below left] {$u'$} ;
\draw[->] (-1,1) --(0,0) node[pos=.5, below left] {$w$} ;
\draw[->] (-3,3.5) -- (-1,3) node[pos=.5, above right] {$u_1$} ;
\draw[->] (-1,3) -- (1,3) node[pos=.5, above] {$w$} ;
\draw[->] (1,3) --(3,3.5) node[pos=.5, above left] {$v_1$} ;
\draw[<-] (2,2) -- (1,3) node[pos=.5, right] {$u'$};
\draw[->] (-2,2) -- (-1,3) node[pos=.5, left] {$v'$};
\draw[->] (-1,1) -- (1,1) node[pos=.5, above] {$w$};
\end{tikzpicture}
\end{figure}

\begin{proof}
  Let
  \begin{align*}
    u&=u_1\cdot \bar u'&
    v'\cdot \bar v_1&=v
  \end{align*}
  be a fine decomposition as in Theorem \ref{T:einfach_amador}. Apply
  Proposition \ref{P:proper_absorbed} to $\bar u'$ and $\bar v_1$ to
  obtain
  \begin{align*}
    \bar u'&=u'\cdot w&
    w\cdot v_1&=\bar v_1,
  \end{align*}
  where $u'$ is properly left-absorbed by $v_1$, $w$ commutes with
  $u'$ and $w$ is a commuting word.

  Uniqueness follows similarly.
\end{proof}

In order to describe canonical paths between elements (or
rather, between flags) in the
Fra\"iss\'e limit $M^N_\infty$, we require a stronger form
of reduction, since applying twice the same operation $\alpha_s$ does
not necessarily yield a global application of $\alpha_s$, but rather a
finite product of proper subletters.

\begin{definition}\label{D:strong_reduction}
  The word $u$ is \emph{strongly
    reduced to} $v$, denoted by $u\str v$, if $v$ is obtained from $u$
  by finitely many iterations of \textsc{Cancellation},
  \textsc{Commutation}, and
  \begin{description}
  \item[\sc Splitting] Replace an occurrence of $s\cdot s$ by a
(possibly trivial) product $t_1\cdots t_n$ of letters $t_i$, each
of which is properly contained in $s$.
  \end{description}
  If $v$ is reduced, we call $v$ a \emph{strong reduct} of $u$.
\end{definition}

\noindent As an example note that $u\cdot u\inv\str 1$.

Despite the possible confusion for the reader, we will  not
refer to reductions defined in \ref{D:reduction} as \emph{weak}
reductions.

Related to the notion of strong reduction, we also consider the
following partial ordering on
words.

\begin{definition}
  For words $u$ and $v$,  we define $u\prec v$ if some
  permutation of $u$ is obtained from $v$ by replacing  at least one
letter $s$ of $v$ by  by a (possibly empty) product of  proper
subletters of $s$. By $u\preceq v$, we mean $u\prec v$ or $u\approx
v$.
\end{definition}

\begin{lemma}\label{L:kleiner_fundiert}\par\noindent
  \begin{enumerate}
  \item $\prec$ is transitive and well-founded.
  \item\label{L:kleiner_fundiert:1} $u'\approx u\prec v\approx v'$
implies $u'\prec v'$.
  \item\label{L:kleiner_fundiert:2} If the strong reduction $u\str v$
involves at least one
    cancellation or splitting, we have $v\prec u$.
  \end{enumerate}
\end{lemma}

\noindent Well-foundedness implies in particular that if $u\prec v$,
then $u\not\approx v$. Furthermore, property
(\ref{L:kleiner_fundiert:1}) yields that $\prec$ induces a partial
order on $\cox$, setting $[u]\prec[v]$ if $u\prec v$, where both $u$
and $v$ are reduced. With this notation, the trivial word $1$ becomes
the smallest element.

\begin{proof}

  To see that $\prec$ is well-founded, we introduce an ordinal-valued
  rank function $\ord$. For $i$ in $[0,N]$, set $\ord_i(w)$ to be
  number of letters $s$ in $w$ with $i+1$ elements. Define now
  \[\ord(w)=\omega^N\ord_N(w)+\omega^{N-1}\ord_{N-1}(w)+\dotsc+\ord_0(w).\]
  Then $u\prec v$ implies $\ord(u)<\ord(v)$.
\end{proof}

The semigroup $\cox$, equipped with the order function as above, is
an ordered semigroup in which left and right-cancellation are
(almost) order-preserving.

\begin{lemma}\label{L:OrdnungMonoid}
  Let $w\cdot v$ be reduced and $w\cdot v\preceq w\cdot v'$. Then
  $v\preceq v'$.
\end{lemma}

The condition that $w\cdot v$ is reduced is needed, by
taking $v'=t \subsetneq s=w=v$ and $w \cdot v  \str 1$.

\begin{proof}
  By induction on the number of letters appearing in $w$, we need only
consider the case where $w=s$ for some interval $s$.

  The assumption implies that $s\cdot v$ is equivalent to a word
  $u_s\cdot u'$ where $u_s\preceq s$ and $u'\preceq v'$. The word
  $u_s$ either equals $s$ or is a product of proper subletters of
  $s$. If $u_s=s$, we have $v\approx u' \preceq v'$ and are
  done. Otherwise, since $s\cdot v$ is reduced, it follows that
  $u_s=1$. This implies $v\prec s\cdot v\approx u'\preceq v'$.
\end{proof}
\begin{cor}\label{C:OrdnungMonoid}
  Given reduced words $w\cdot v$ and $v'$ such that $w\cdot v$ is
  smaller than some strong reduct of $w\cdot v'$, then $v\preceq v'$.
\end{cor}

\begin{lemma}
  The partial order $\preceq$ is compatible with the semigroup
operation
in $\cox$.
\end{lemma}
\begin{proof}
  Given reduced words $u$,$v$ and $w$, we have to show the following:
  \begin{align*}
    [u]\preceq[v]&\;\;\Rightarrow\;\;[w][u]\preceq[w][v]\\
    \intertext{and}
    [u]\preceq[v]&\;\;\Rightarrow\;\;[u][w]\preceq[v][w].\\
  \end{align*}
 By symmetry, it is sufficient to show the first
implication. By induction on $|w|$, it is enough to
consider the case where $w$ is a single letter $s$.

Suppose first that $s$ is left-absorbed by $v$. By Corollary
\ref{C:absorb_cox},
  $$[s][v]=[v].$$  If $s$ is also left-absorbed by $u$, we are
clearly done. Otherwise, by Theorem \ref{T:einfach_amador},  decompose
$u$ (up to permutation)
as $u=u'\cdot u_1$, where $s\cdot u_1$ is the reduct of
$s\cdot u$. Also, write $v=\bar v\cdot t\cdot v_1$ such that
$s\subset t$ and $\bar v$ is in $\Cent(s)$. Now, the word $u_1\preceq
u\preceq v$, so write $u_1=\bar u_1\cdot u_1^t \cdot\bar u_1^1$, where
$\bar u_1\preceq \bar v$, $ u_1^t\preceq t$ and $ u_1^1\preceq v_1$.
Since $s\cdot u_1$ is reduced, so is $s\cdot \bar
u_1\cdot u_1^t = \bar u_1\cdot s\cdot u_1^t$.

This forces $ u_1^t$ to be either trivial or different from $t$ (and
$s\neq t$ as well).  In both cases, we have that $s\cdot u_1^t\preceq
t$, which implies $s\cdot u_1\preceq v$, so we are done.

If  $s$ is not left-absorbed by $v$, by Theorem
\ref{T:einfach_amador}, we can write (up to permutation)  $v=v'\cdot
v_1$, where $v'$ is properly absorbed by $s$ and $s\cdot v_1$ is
reduced. So $[s][v]=[s\cdot v_1]$. If $s$ is
  left-absorbed by $u$, then $$[s][u]=[u] \preceq
  [v'\cdot v_1]\prec [s\cdot v_1].$$
\noindent Otherwise, write $u=\bar u\cdot u'\cdot u_1$ as above such
that $s\cdot u\to \bar u\cdot s\cdot u_1$.
  Since $\bar u$ and $s$ commute, note that $\bar u\cdot
u_1$ is irreducible, since $u$ is. Decompose $\bar u\cdot
u_1=u_1'\cdot u_{11}$ with $u_1'\preceq v'$ and
  $u_{11}\preceq v_1$. Since $s\cdot \bar u\cdot u_1=\bar u\cdot
s\cdot u_1$ is reduced, the word $u_1'$ must be trivial.
Therefore  $s\cdot \bar u \cdot u_1= s\cdot u_{11}\preceq
s\cdot v_1$.
\end{proof}

In particular, since $1\preceq v$ for any word $v$, we obtain the
following result.

\begin{cor}\label{C:u_kleinerals_uv}
  Let $u$ be reduced. Given any word $v$, the reduction $w$ of
$u\cdot v$ is $\preceq$-larger than $u$.
\end{cor}

In contrast to Proposition \ref{P:nonsplitred}, uniqueness of strong
reductions does no longer hold, \emph{e.g.}  $s\cdot s\str s$
and $s\cdot s\str 1$. However, we get the following result, which
allows us to permute the steps of the strong reduction:

\begin{prop}[Commutation Lemma]\label{P:commutation}
  If  $x$ is a strong reduct of $u\cdot v\cdot w$, then
  there is a strong reduct $y$  of $v$ such that $u\cdot y\cdot
  w\str x$.
\end{prop}
\begin{proof}
  Consider first the case where $u=t$ has length
$1$, the word $v$ has length $2$ and $w$ is empty. Suppose
furthermore that in the first step of the reduction $t\cdot
  v\str x$, the letter $t$ is deleted. It is easy to check that
setting $y$ as the reduct of $v$, the results follows, except if
  $v=s\cdot s$, the letter $t$ is contained in $s$ and the strong
  reduction is $t\cdot (s\cdot s)\str s\cdot s\str x$, where $x$ is a
  product of letters which are properly contained in $s$. Then:

 \begin{itemize}
   \item If $t=s$, set $y=s$.
   \item If $t\cdot x\str x$, set $y=x$.
   \item Otherwise, apply  Theorem  \ref{T:einfach_amador} to $x$
and $t$ and decompose $x=x'\cdot x_1$ such that $|x'|$ is properly
contained in $t$ and $t\cdot x_1$ is reduced. Set $y=t\cdot x_1$.
  \end{itemize}

 In all three cases, the
  strong reductions hold:
  \[t\cdot (s\cdot s)\str t\cdot y\str x.\]

In order to show the proposition for the general case,
motivated by the proof of \ref{P:nonsplitred},
let us introduce the following rule:
\begin{description}
    \item[\sc Generalised Splitting] Given a word $s_1\cdots s_n$ and
      a pair of indices $i\not=j$ such that $s_i=s_j$ and $s_i$
      commutes with all $s_k$'s with $k$ between $i$ and $j$, delete
      $s_j$ and replace $s_i$ by a product of letters which are
      properly contained in $s$.
  \end{description}

  Note that a strong reduction consists of finitely
many generalised cancellations and generalised splittings, followed by
 commutation (if needed).

If $v$ is reduced, set  $y=v$. Otherwise, we will apply induction on
the $\prec$-order type of $v$. Suppose therefore that the assertion
holds  for all $v'\prec v$ and consider $x$ a strong reduct of
$u\cdot v\cdot w$. If $2<|v|$, then (after permutation) write
$v=v_1\cdot a\cdot v_2$, where $a$ is a non-reduced word of length
$2$. Note that by assumption, the subword $a\prec v$, so there is a
strong reduct $b$ of $a$ such that $u\cdot v_1\cdot b\cdot v_2 \cdot
w\str x$. Since $a$ is not reduced, we have $b\prec a$ and
thus $v_1\cdot  b\cdot v_2\prec v$. Induction yields the existence of
a strong reduct $y$ of $v_1\cdot b\cdot v_2$ such that
$$u\cdot y\cdot w\str x.$$

Note that $v= v_1\cdot a\cdot v_2\str
v_1\cdot b\cdot v_2\str y$. Therefore, we may assume that $v$
has length $2$ and it is non-reduced. By the above discussion,
the first step in the strong reduction

$$ u\cdot v \cdot w \str x.$$

\noindent is either a generalised cancellation or a generalised
splitting. If it involves only letters from $v$, its strong reduction
is $\preceq$-smaller and one step shorter to the output $x$, so we are
done by induction on the number of steps in the strong reduction.
Likewise if the letters involved are in $u\cdot w$. Thus, we may
assume that there are two letters  $t$ and $r$ witnessing the
reduction in the first step and, say, the letter $t$ occurs in $u$ and
$r$ in
$v$.

We have two cases:

\begin{itemize}
 \item The letter $t$ is absorbed by $v$. In particular, the letter
lies in the final segment $\tilde u$. Write $u=u_1\cdot t$.  If it was
a generalised splitting, the result $v'\prec v$ and $u_1\cdot v'
\cdot w\str x$. Induction gives a strong reduct $x'$ of $v'$ such that
$u_1\cdot x'\cdot w\str x$. In particular, we are now in the case
$t\cdot v \str x'$ and thus, by the discussion at the beginning of
the proof, there exists
a strong reduction $y$ of $v$ such that $t\cdot y\str x'$. Note that
$$u\cdot v \cdot w = u_1\cdot (t\cdot v)\cdot w \str u_1 (t\cdot
y)\cdot w \str u_1\cdot x'\cdot w\str x,$$
so we are done.

If the first step was a generalised cancellation, the word $v$ does
not change and now $u_1\cdot v\cdot w\str x$ in one step less. We
obtain a strong reduct $x'$ of $v$ with $u_1\cdot x'\cdot w\str x$.
Again, note that $t\cdot v\str v\str x'$ so, again by the previous
discussion, there is a strong reduct $y$ of $v$ which does the job.

   \item Otherwise, the occurrence $r$ in $v$ is deleted. If $r=t$, we
are in the previous case. Suppose hence $r\subsetneq t$ and write
$u=u_1\cdot t\cdot u_2$, where $u_2$  commutes with $r$. We may
assume that $v=r\cdot s$. Note that $r$ and $s$ are comparable, since
$v$ is not reduced. If $r\subseteq s$, then set $y=s$, which is a
strong reduct of $v$. We have that $u\cdot y \cdot w\str x$.

If $s\subsetneq r$, then $s$ and $u_2$ commute as well. Note that
$u_1\cdot (t\cdot s)\cdot u_2\cdot w=u \cdot s\cdot w\str x$
in one step less. We have that  $u_1\cdot t\cdot
u_2\cdot w \str x$ and setting $y=r$ does the job.
\end{itemize}
\end{proof}

Despite the apparent arbitrarity of the strong reductions, they are
orthogonal to the reduction without splitting, as the following
result shows.

\begin{prop}\label{P:strongvszerfallos}
  Let $u$ and $v$ be reduced words and consider $x$ the reduct of
$u\cdot v$ and $x^*$ some strong reduct of $u\cdot v$, where splitting
occurs.  Then  $x^*\prec x$.
\end{prop}

Note that that this is not true for the product of three reduced
words: $s\cdot s\cdot s$ can be strongly reduced to $s$ by one
splitting operation.

\begin{proof}
  Remark first that, if $w=s_1\cdots s_n$ is a
  commuting word and $y^*$ is a strong reduct of $w\cdot w$, then
  $y^*=t_1\cdots t_n$, where each $t_i$ is a strong reducts of
$s_i\cdot s_i$.  If splitting ever occured in the reduction, then
$y^*\prec  w$.

  To prove the proposition, choose decompositions $u=u_1\cdot u'\cdot
w$ and  $w\cdot v'\cdot v_1=v$, as in Corollary \ref{C:amador}.  A
general  cancellation applied to $u_1\cdot u'\cdot w\cdot w\cdot
v'\cdot v_1$
  does the following: either the last letter of (a permutation of)
  $u'$ is deleted, the first letter of $v'$ is deleted or one letter
in one of the copies of $w$ is deleted. Hence,  after
  finitely may generalised cancellations, the end result has the form
$z=u_1\cdot  u''\cdot w'\cdot w'\cdot v''\cdot v_1$, where $u''$ is a
left end of $u'$, the subword $v''$ is a right right end of $v'$ and
$w'$ is a subword of  $w$. A generalised splitting for $z$ can
only happen inside $w'\cdot w'$. So we obtain a word $z'=u_1\cdot
  u''\cdot a \cdot v''\cdot v_1$, where $a$ is obtain from $w\cdot w$
  by the splitting operation. If we apply the Commutation Lemma
\ref{P:commutation} to  $(u_1\cdot v')\cdot a\cdot (u'\cdot
v_1)\approx z'$, we obtain a strong reduct $b$ of $a$ such that
$u_1\cdot b\cdot v_1\str x^*$. The above observation gives that
$b\prec w$ and thus $x^*\preceq u_1\cdot b\cdot v_1\prec
  u_1\cdot w\cdot v_1\approx x$.
\end{proof}

Inspired by the following picture:

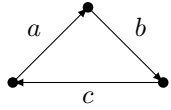
\begin{figure}[!htbp]
\centering

\begin{tikzpicture}[>=latex,->]

\fill (-1,0) circle (2pt);
\fill (0,1) circle (2pt);
\fill (1,0) circle (2pt);

\draw[->] (-1,0) -- (0,1) node[pos=.5, above left] {$a$}
;
\draw[->] (0,1) -- (1,0) node[pos=.5, above right] {$b$} ;
\draw[->] (1,0) -- (-1,0) node[pos=.5, below] {$c$} ;

\end{tikzpicture}

\end{figure}

\noindent we deduce strong reductions from a given one, as long as
products are involved.

\begin{prop}[Triangle Lemma]\label{P:triangle}
  Let $a$, $b$ and $c$ be reduced words. Then $a\cdot b\str c\inv$
  implies $c\cdot a\str b\inv$ and $b\cdot c\str a\inv$.
\end{prop}
\begin{proof}
  By symmetry, it is enough to show that $a\cdot b\str c\inv$
 implies $c\cdot a\str b\inv$. Suppose hence that $a\cdot b\str
c\inv$. We apply induction on the $\prec$-type of $a$ and $b$.

If $a\cdot b$ is reduced, then $c=b\inv\cdot a\inv$ and
so $c\cdot a=b\inv\cdot a\inv\cdot a\str b\inv$. Thus, assume
$a\cdot b$ is not reduced. We distinguish the following  cases (up to
permutation):

\begin{itemize}
 \item  $a= a_1\cdot s$, where $s$ is properly left-absorbed
    by $b$. Since $b$ is the only strong reduct of $s\cdot b$, the
Commutation Lemma \ref{P:commutation} gives that
$$a\cdot b=a_1\cdot(s\cdot b)\to a_1\cdot b\str c\inv.$$

\noindent Since $a_1\prec a$, induction gives that $c\cdot a_1\str
b\inv$, which implies that

$$c\cdot a=(c\cdot a_1)\cdot s\str b\inv\cdot s\to b\inv.$$

\item $b=s\cdot b_1$, where $s$ is properly right-absorbed
    by $a$. Again $a\cdot b=a\cdot(s\cdot b_1)\to a\cdot b_1\str
c\inv$, so by induction $c\cdot a\str b_1\inv$. Thus
$$c\cdot(a\cdot s)\str b_1\inv\cdot s=b\inv.$$
\noindent Since $a$ is the only  strong reduct of $a\cdot s$, again
Proposition \ref{P:commutation} gives that $c\cdot a\str b\inv$.

\item $a=a_1\cdot s$ and $b=s\cdot b_1$ Since $a_1\cdot(s\cdot s)\cdot
b_1\str c\inv$, Proposition \ref{P:commutation} provides a strong
reduct $x$ of $s\cdot s$ such that $a_1\cdot x\cdot b_1\str c\inv
  b\str c\inv$. The word $x$ is either $s$ or a product
of proper subletters of $x$ and hence $\prec$-smaller than $s$. Since
$b=s\cdot b_1$ is reduced, apply Theorem \ref{T:einfach_amador} to
decompose $x=x_1\cdot x'$, where $x'$ is properly left absorbed by
$b_1$ and $x_1\cdot b_1$ is reduced (If $x=s$, then $x_1=s$ and
  $x'=1$). Since $x'\cdot b_1\str b_1$, the reduction $(a_1\cdot
x_1)\cdot(x'\cdot b_1)\str c\inv$ implies $ a_1\cdot x_1\cdot
b_1\str c\inv$. Since $a_1\prec a$ and $x_1\cdot b_1\preceq b$,
induction gives that
$$c\cdot a_1\str b_1\inv\cdot x_1\inv.$$

\noindent In particular,
$$c\cdot a= c\cdot a_1 \cdot s \str (b_1\inv\cdot x_1\inv)\cdot s
\to b_1\inv \cdot s \to b\inv.$$
\end{itemize}
\end{proof}
We can now easily conclude the following:
\begin{cor}\label{C:word_inverses}
  If $u$ and $v$ are both reduced and $u\cdot v\str 1$, then
$v\approx u\inv$~.
\end{cor}
\begin{proof}
  The Triangle Lemma (Proposition \ref{P:triangle}) yields $1\cdot
u\str v\inv$ and $v\cdot 1\str u\inv$. That is,  $u\inv\str v$ and
$v\str u\inv$. Thus
  $$u\inv\preceq v\preceq u\inv,$$
\noindent and therefore $v\approx u\inv$~.
\end{proof}

Recall by Corollary
\ref{C:absorb_cox} that if $u$ is the reduct of $u\cdot v$, then
$v$ is right-absorbed by $u$. This is no  longer
true for strong reductions: take for example
$$(s\cdot t)\cdot(t\cdot s\cdot t) = s \cdot (t\cdot t) \cdot (s\cdot
t) \str s \cdot (s\cdot t)\str s\cdot t.$$

\noindent However, in  certain situations we are still able to
conclude the same  for strong reductions as for reductions with no
splitting.

\begin{lemma}
  Let $u$ and $v$ be reduced. If every letter in $v$ which is
  right-absorbed by $u$ is properly absorbed and  $u\cdot v\str
u$, then $u\cdot v\to u$.
\end{lemma}
\begin{proof}
  Apply Theorem \ref{T:einfach_amador} to obtain fine decompositions
$u=u_1\cdot u'$ and $v'\cdot v_1=v$ such that $u'$ is properly
left-absorbed by $v_1$, the word $v'$ is right-absorbed by $u_1$, the
words $u'$ and $v'$ commute and $u_1\cdot v_1$ is reduced.

By hypothesis, the word  $v'$ is properly right-absorbed by
$u_1$. The Commutation Lemma \ref{P:commutation}  applied to
$(u_1\cdot v')\cdot(u'\cdot v_1)\str u$ gives
$$(u_1\cdot v')\cdot(u'\cdot v_1)\to u_1\cdot v_1\str u.$$

\noindent Since $u_1\cdot v_1$ is reduced, we have $u_1\cdot v_1=u$.
So $v_1=u'$ must properly absorb itself, which is a contradiction
unless $v_1=1$ and thus $u\cdot v\to u$.
\end{proof}

Let us conclude by giving a criteria for when a word wobbles inside
two other. This will be useful for determining all possible paths
between two given flags.

\begin{prop}\label{P:wortwobbel}
  Let $u\cdot v$ and $w$ be reduced. If $u\cdot w\str u$ and
  $w\inv\cdot v\str v$, then $|w|\subset \wob(u,v)$.
\end{prop}

\begin{proof}
  By Remark \ref{R:wob_echt}, it is enough to prove that $w$ is
properly right-absorbed by $u$ (and likewise for $v$). We proceed by
induction on the length of $|v|$.

 If $v=1$, then $ w\inv\cdot 1\str 1$ implies $w\inv=1$, since $w$ is
reduced.

 Suppose now that $v=s\cdot v_1$. Set $u\cdot s=u_1$, which is again
reduced. So is  $u_1\cdot v_1=u\cdot v$.

The condition $w\inv\cdot v\str v$ implies
  $v\inv\cdot w\str v\inv$ by Proposition \ref{P:triangle}. This
implies
$$v_1\inv\cdot(s\cdot  w\cdot s)\str (v_1\inv\cdot s)\cdot s\str
v_1\inv.$$

By the Commutation Lemma (Proposition \ref{P:commutation}), there is
a strong reduct $w_1$ of $s\cdot w\cdot s$ with $ v_1\inv\cdot
  w_1\str v_1\inv$, or equivalently,  $w_1\inv\cdot v_1\str v_1$.

The Triangle Lemma \ref{P:triangle} gives that  $s\cdot (w\cdot
s)\str w_1$ implies
$ w_1\inv\cdot s\str s\cdot w\inv$, that is, $s\cdot w_1\str
  w\cdot s$.

In particular, we have that $u_1\cdot w_1=u\cdot (s\cdot w_1)\str
u\cdot (w\cdot s)\str u\cdot s=u_1$.

  By the induction hypothesis applied to $u_1$, $v_1$ and $w_1$, we
have that $w_1$ is properly right-absorbed by
  $u_1=u\cdot s$. By Lemma \ref{L:schlucktheorie}
(\ref{L:schlucktheorie:3}), write $w_1$ as $w_s\cdot w_u$
  where $w_s$ is properly absorbed by $s$ and $w_u$ is properly
  right-absorbed by $u$ and commutes with $s$. Note that $s\cdot w_u$
is the only strong reduct of $s\cdot w_1$. Proposition
\ref{P:commutation} yields that the strong reduction $(s\cdot
w_1)\cdot s\str w\cdot s\cdot s\str w$ factors through $s\cdot
w_u\cdot s\str w$.

 Since  $s\cdot w_u\cdot s$ is equivalent to $s\cdot s\cdot w_u$,
there is strong reduct $x$ of $s\cdot s$ such
  that $x\cdot w_u\str w$. However, the product $x\cdot w_u$ is
already reduced and so $x\cdot w_u=w$. The reduct $x$ is either $s$
or consists of proper subletters of $s$. Suppose that $x=s$. Then
$u\cdot w=u\cdot s\cdot w_u=u\cdot s$, since $w_u$ is properly
right-absorbed by $u$ and commutes with $s$. This contradicts with
$u\cdot w\str u$. Hence, the word  $x$ consists of proper subletters
of $s$. By Theorem \ref{T:einfach_amador}, since $u\cdot s$ is
reduced, decompose $x$ into $x'\cdot x_1$, where $x'$ is properly
right-absorbed by $u$ and $u\cdot x_1$ is reduced. Then $u\cdot x_1$
is the only strong reduct of $u\cdot w=u\cdot x'\cdot x_1\cdot  w_u$.
We conclude that $u\cdot x_1=u$ and thus $x_1=1$ by Corollary
\ref{C:absorb_cox}. Hence, the word $w=x'\cdot  w_u$ is properly
right-absorbed by $u$.
\end{proof}
\section{Flags and Paths}

Let $M$ be any colored $N$-space. As in Definition
\ref{D:complete}, recall that a \emph{flag} $F$ in $M$ is a path
$a_0-\ldots-a_N$ of length $N$, where each $a_i$ belongs to
$\A_i(M)$. We call $a_i$ the $i$-\emph{vertex} of the flag
$F$.

\begin{definition}\label{D:weakop}
  Given flags $F$ and $G$, we say that $G$ is obtained from $F$ by
  the \emph{weak operation} $\alpha_s$ if $s$ consists of the indexes
  where the vertices of $F$ and $G$ differ. A weak path of flags $P$
  is a sequence of flags $F_0,\ldots,F_n$, where each $F_i$ is
obtained from $F_{i-1}$ by a weak operation $\alpha_{s_i}$. We call
  $s_1\cdots s_n$ the \emph{word of $P$.}
\end{definition}
More generally, we define:
\begin{definition}\label{D:Eq}
   Let $A$ be a subset of $[0,N]$. Two flags are \emph{equivalent
    modulo} $A$ if they have the same vertices in all levels outside
  $A$. We write $F/A$ for the equivalence class of $F$ modulo $A$.
\end{definition}
\noindent Note that $F/A$ is interdefinable with the set of vertices
of $F$ with levels outside $A$. For $i$ in $[0,N]$ and
$A=[0,N]\setminus \{i\}$, the equivalence class $F/A_i$ is
interdefinable with the vertex $f_i$. We can say that $F/A_i$
and $F'/A_j$, for $i$ and $j$ immediate succesors, are connected in
case they belong to a class of a common flag $G$. This induces a
structure bi-interpretable with $PS_N$.

Any two flags can be connected by a weak flag path:  decompose the
set $I$ of indices where the vertices of $F$ and $G$ differ
as the disjoint union $s_1\cup\cdots s_n$ of intervals, such that
$s_i$ and $s_j$ commute for $i\neq j$. Then $F$ and $G$ are connected
by a weak path with word $s_1\cdots s_n$. In particular, we obtain the
following.

\begin{lemma}\label{L:Eq_flags}
  Two flags $F$ and $G$ are equivalent modulo $A$ \iff they can be
connected by a weak path whose word consists of letters contained in
$A$. Furthermore,  there is such a path whose word is
commuting.

\noindent In particular, any two flags are connected by a weak path,
by taking $A=[0,N]$.
\end{lemma}

Commuting letters in a path induces another path whose word
is a permutation of the previous one.

\begin{lemma}\label{L:vertauschung}
  Let $s$ and $t$ be commuting letters and assume that $F$ and $G$ are
  connected by a weak flag path with word $s\cdot t$. Then there is a unique
  weak flag path from $F$ to $G$ with word $t\cdot s$.
\end{lemma}

\begin{proof}
  Given the path $F - H - G$ with word $s\cdot t$, define a new flag
$H'$ by replacing the $s$-part of $H$ by the $s$-part of $F$ and its
$t$-part by the $t$-part of $G$. By construction, the weak path
$F - H' - G$ has word $t\cdot s$.

\noindent Uniqueness is clear since the $s$-part and the $t$-part of
$H'$ are determined by those of $F$ and $G$.
\end{proof}

Iterating the previous result, since any permutation can be achieved
by a sequence of transpositions of adjacent commuting letters, given a
weak path $P_u$ be a from $F$ to $G$ with word $u$, if $v$ is a
permutation of $u$, we can connect $F$ and $G$ by a weak path $P_v$
with word $v$. Note that $P_v$ does not depend on the sequence of
transpositions and the collection of vertices of flags occuring in
$P_u$ agrees with the one of flags in $P$. We call the path $P_v$ a
\emph{permutation of} $P_u$.

We will now link the words appearing in weak paths with their
distance as in Lemma \ref{L:operation}.

\begin{lemma}\label{L:abstand_flaggenpfad}
  Let\/ $t=(l,r)$ and $F$ and $G$ be equivalent modulo $t$. Let $a_l$
  and $a_r$ the vertices of $F$ (and $G$) of level\/ $l$ and $r$,
  respectively. Given a subletter $s\subset t$, the following are
  equivalent:
  \begin{enumerate}[a)]
  \item\label{L:abstand_flaggenpfad:abstand} The flags $F$ and $G$
    have finite $s$-distance in $M^{a_r}_{a_l}$.
  \item\label{L:abstand_flaggenpfad:pfad} The flag $F$ and $G$ are
    connected by a weak flag path whose letters are contained in $t$
    but do not contain $s$.
  \end{enumerate}
\end{lemma}
\begin{proof}
  $(\ref{L:abstand_flaggenpfad:abstand})\to
  (\ref{L:abstand_flaggenpfad:pfad})$: Consider a path $ b_0,\ldots
  b_n$ in $\A_s(M^{a_r}_{a_l})$ connecting two vertices of $F$ and
  $G$.  For every $i$ in $\{1,\ldots,n-1\}$, pick a flag $F_i$
  containing $b_i$ and $b_{i+1}$ which agrees with $F$ and $G$ outside
  the levels in $t$.  Set $F_0=F$ and $F_{n}=G$. If $b_{i+1}$  has
  level $j_i$, then $F_i$ and $F_{i+1}$ are equivalent modulo
  $t\setminus\{j_i\}$. They are thus connected by a weak flag path
  whose letters are contained in $t\setminus\{j\}$ and therefore none
  contains $s$. The concatenation of these flag paths gives the
  result.

  $(\ref{L:abstand_flaggenpfad:pfad})\to
  \ref{L:abstand_flaggenpfad:abstand})$: Let $F=F_0 - \ldots - F_n=G$
  be a weak flag path whose letters are in $t$ but do not contain
  $s$. For every $i$ in $\{0,n-1\}$, the flags $F_i$ and
  $F_{i+1}$ have a common vertex in $\A_s(M^{a_r}_{a_l})$. Thus, we
  can connect $F$ and $G$ by a path whose vertices lie in
  $\A_s(F_0)\cup\ldots\cup\A_s(F_n)$ and hence, between $a_l$ and
  $a_r$.
\end{proof}

In order to distinguish between weak operations between flags and
global applications of $\alpha_s$ to nice sets, as in Lemma
\ref{L:nice}, we introduce the following definition, at the level of
flags.

\begin{definition}
 For $s=(l,r)$, the flag $G$ is obtained by a \emph{global application
of} $\alpha_s$ from  $F$ if $G$ is obtained by a weak application of
$\alpha_s$ from $F$  and its new vertices  have infinite
distance in $M^{a_r}_{a_l}$ from $F$, where $a_l$ and $a_r$
are the vertices of of $F$ (and $G$) of level $l$ and
$r$, respectively.
\end{definition}

Since a flag is in particular a nice set, these two
definitions agree, by applying Lemma \ref{L:abstand_flaggenpfad} to
the case $t=s$:

\begin{cor}\label{C:fteprod}
  Given an interval $s$ and flags $F$ and $G$, the following
  are equivalent:
  \begin{enumerate}[a)]
  \item The flag $G$ is obtained from $F$ by a global application of
    $\alpha_s$, as in Lemma \ref{L:nice}.
  \item The flag $G$ is obtained from $F$ by the weak operation
$\alpha_s$ and there is no weak flag path connecting them whose
word consists of proper subletters of $s$.
  \end{enumerate}
\end{cor}

\begin{definition}
  A \emph{flag path} is a weak flag path where each flag is
obtained from its predecessor by a global operation. If $F$ and $G$
are connected with a flag path with word $u$, we write
\[F\lmto{u}G.\]
  A flag path is \emph{reduced} if its word is reduced.
\end{definition}

\begin{lemma}\label{L:stark_aus_schwach}
  If there is a weak path from $F$ to $G$ with word $u$, we have
  $F\lmto{v} G$ for some $v$ with $v\preceq u$.
\end{lemma}
\begin{proof}
  By Lemmma \ref{L:Eq_flags}, choose a weak path
$F=F_0-\ldots- F_n=G$ whose word $v=s_1\cdots s_n$ is
$\preceq$-smaller to $u$ and minimal such. We need only show
that this path is a flag path.
Otherwise, some operation  $\alpha_{s_i}$ is not global and, by
Corollary \ref{C:fteprod}, we can connect $F_{i-1}$ and $F_i$ with a
weak path whose word consists of proper subletters of $s_i$. The
resulting word is $\prec$-smaller than $v$, contradicting its
minimality.
\end{proof}

Combining the previous result and  Corollary \ref{C:fteprod}, we
obtain the following:

\begin{cor}\label{C:stark_oder_zerfall}
  If $F$ and $G$ are equivalent modulo $t$, then either $F\lmto{t} G$
or  $F\lmto{x}G$, for some product $x$ whose factors are proper
subletters of $t$.
\end{cor}

\begin{proof}
  By Lemma \ref{L:Eq_flags}, the flag $G$ is obtained from $F$ by a
weak path $P$ whose word $x$ either equals $t$ or consists of letters
properly contained in $t$. By Lemma \ref{L:stark_aus_schwach}, we may
assume  that $P$ is a flag path.
\end{proof}

We can now compose flag paths, using the results of the previous
section.

\begin{lemma}\label{L:flag_rules}
  Assume $F\lmto{s}G\lmto{t}H$.
  \begin{enumerate}
  \item\label{L:flag_rules:comm} If $s$ and $t$ commute, there is a
    unique $G'$ with $F\lmto{t}G'\lmto{s}H$.
  \item\label{L:flag_rules:echt} If $s$ is a proper subset of $t$, then
    $F\lmto{t}H$. Similarly, if $t$ is a proper subset of $s$, then
    $F\lmto{s}H$.
  \item\label{L:flag_rules:split} If $s=t$, then either $F\lmto{t}H$
    or $F\lmto{x}H$, for some product $x$ whose factors are proper
subletters of $t$.
  \end{enumerate}
In particular, a permutation of a flag path yields again a flag
path, by $(\ref{L:flag_rules:comm})$.
\end{lemma}

\begin{proof}
  Property $(\ref{L:flag_rules:comm})$ follows
easily from Lemma \ref{L:vertauschung}, since the permutation of a
reduced word remains reduced.

\noindent For $(\ref{L:flag_rules:echt})$,  assume $s\subsetneq t$.
Then $H$ is equivalent to $F$ modulo $t$. So by Corollary
\ref{C:stark_oder_zerfall}, either  $F\lmto{t}H$ or $F\lmto{x}H$,
where $x$ consists of proper subletters of $t$. The latter implies
that $G\lmto{s\cdot x}H$, which contradicts the
assumption $G\lmto{t}H$. The proof is similar if $t$ is a proper
subset of $s$.

  Property $(\ref{L:flag_rules:split})$ clearly follows from
Corollary \ref{C:stark_oder_zerfall}, as $F$ and $H$
are equivalent modulo $t$.
\end{proof}

Lemma \ref{L:stark_aus_schwach} yields the following.

\begin{cor}\label{C:minimal_path_is_reduced}
  Let $F$ and $G$ be two flags.
  \begin{enumerate}
  \item If $F\lmto{u}G$, then $F\lmto{v}G$ for some strong reduct
$v$ of $u$.

  \item If $u$ is $\prec$-minimal with $F\lmto{u}G$, then $u$ is
    reduced.
  \end{enumerate}
\end{cor}

\begin{definition}
  Let $A$ be a subset of $M$ and two vertices $a_l$ and $a_r$ in $A$
such that $a_l$ lies below $a_r$ in $A$. The pair $(a_l,a_r)$ is
called  \emph{open} in $A$ if there are vertices $b$
  and $c$ in $A^{a_r}_{a_l}$ whose distance in
  $M^{a_r}_{a_l}$ is infinite.

A pair as before which is not open is called \emph{closed}.
\end{definition}

\begin{lemma}\label{L:nice_entspannt}
  Let $s=(l,r)$ be an interval and $M$ be simply connected. Take a
nice subset $A$ of $M$ with two distinguished vertices $a_l$ and
$a_r$ of
levels $l$ and $r$, respectively. Given a flag $F$ in $A$ containing
$a_l$ and $a_r$, assume that $F\lmto{s}G$ for some flag $G$ in $M$.
Set $B=A\cup G$. If the pair
$(a_l,a_r)$ is closed in $A$, we have that:
  \begin{enumerate}
  \item The set $B$ is obtained from $A$ by a global application of
$\alpha_{s}$ on $(a_l,a_r)$.
  \item The open pairs in $B$ are exactly the open pairs of $A$
together with $(a_l,a_r)$.
  \end{enumerate}
\end{lemma}
\begin{proof}
  For the first assertion, by Lemma \ref{L:nice}, we need only
check that
  \[\dd^{M^{a_r}_{a_l}}(d,A)=\infty,\]
\noindent  where $d$ is  one of the new vertices of $G$.

  Pick any $b$ in $A^{a_r}_{a_l}$ and choose some vertex $c$ in $F$
  between $a_l$ and $a_r$. Since $(a_l,a_r)$ is closed in $A$, we
have that
  $\dd^{M^{a_r}_{a_l}}(b,c)<\infty.$ Since
$F\lmto{s}G$, Lemma \ref{L:abstand_flaggenpfad} shows that 
  $\dd^{M^{a_r}_{a_l}}(c,d)=\infty$. In particular,
  $$\dd^{M^{a_r}_{a_l}}(b,d)=\infty,$$
\noindent which gives the desired result.

  For the second assertion, clearly $(a_l,a_r)$ is now open in $B$. We
need only show there  are no new open pairs in $B$. Consider
an open pair $(x,y)$. If $x$ is one of the new elements of $G$,
then either $y$ is either also in $B\setminus A$ or in $A$ and
either equal to $a_r$ or above of it. If
both $x$ and $y$ lie in $B\setminus A$, they form a closed pair. If
$y=a_r$, all vertices between $x$ and $y$ lie
on $B\setminus A$, and thus the pair $(x,y)$ is closed. If $y$
lies above $a_r$ in $A$, then all vertices between $x$ and $y$ are are
connected with $a_r$ and thus their distance is finite, so  $(x,y)$ is
closed.

Hence, we conclude that both $x$ and $y$ lie in $A$. Suppose
 $(x,y)$ is not $(a_l,a_r)$. Either it was already open in $A$ or
there is a vertex $d$ in $B\setminus A$ whose distance to some $b$ in
$A$ is infinite in $M^y_x$. In particular, the vertex $x$ lies below
$a_l$ and $y$ lies above $a_r$. Since $(x,y)$ is closed in $A$, the
distance between $b$ and $a_l$ in $M^y_x$ is finite and thus $b$ and
$d$ have finite distance in $M^y_x$, which is a contradiction.
\end{proof}

Flag paths provide scaffolds which are nice sets, as the following
Lemma shows.

\begin{lemma}\label{L:scaffold} Let $M$ be simply connected and
  $F_0\lmto{s_1} F_1\lmto{s_2}\ldots\lmto{s_n}F_n$ be a reduced flag
path in $M$.  The following hold:
  \begin{enumerate}
  \item The set $A_n=F_0\cup F_1\cup\ldots\cup F_n$ is nice in $M$.
  \item If $a_0-\ldots-a_N$ are the vertices of $F_n$, then
    $(a_l,a_r)$ is open in $A_n$ \iff the letter $(l,r)$ belongs to
    final segment of $s_1s_2\ldots s_n$.
  \end{enumerate}
\end{lemma}
\begin{proof}
  We prove it by induction on $n$. Let $s_i=(l_i,r_i)$ and
  $w_i=s_1s_2\ldots s_i$. If $n=0$, there is nothing to prove, since
  any flag is nice and the word $w_0$ is trivial.

  Suppose hence that $n>0$ and let $F_n=a_0-\ldots-a_N$.  Since $w_n$
is reduced by assumption, the letter $s_n$ does not belong to the
final segment of
  $w_{n-1}$. Therefore, the pair $(a_{l_n},a_{r_n})$ appeared already
  in $F_{n-1}$ and, by induction, it is closed in $A_{n-1}$, which is
 nice. Lemma \ref{L:nice_entspannt} gives that so is $A_n$.

  Furthermore, Lemma \ref{L:nice_entspannt} also implies that
  $(a_l,a_r)$ is open in $A_n$ \iff $(a_l,a_r)=(a_{l_n},a_{r_n})$ or
  it belongs to $A_{n-1}$ and was already open in $A_{n-1}$. In
  particular, the pair $(a_l,a_r)$ belongs to $A_{n-1}$ \iff either
  $(l,r)$ commutes with $s_n$ or $(l,r)$ contains $s_n$. Since $s_n$
 is not contained in the final segment of $w_{n-1}$, induction gives
  that $(a_l,a_r)$ is open in $A_n$ iff $(l,r)=s_n$ or $(l,r)$
  commutes with $s_n$ and belongs to the final segment of $w_{n-1}$,
 which means that $(l,r)$ belongs to the final segment of $w_n$.
\end{proof}
If the space is simply connected, we shall prove that there are no
flag loops, unless they are not reduced.

\begin{cor}\label{C:geschl_flaggenweg}
  If $M$ is simply connected, there are no non-trivial closed reduced
  flags paths.
\end{cor}
\begin{proof}
  Let $F_0\lmto{s_1} F_1\lmto{s_2}\ldots\lmto{s_n}F_n$ a
non-trivial reduced flag path. By Lemmata \ref{L:nice_entspannt} and
\ref{L:scaffold}, the flag $F_n$ is obtained by a global application
of $\alpha_{s_n}$ to $F_0\cup\dotsb\cup F_{n-1}$. In particular,
the flag $F_n$ must differ from $F_0$.
\end{proof}

Since there are no loops, the reduced word of a flag path is hence
unique, up to permutation.

\begin{prop}\label{P:unique_path_word}
  The word of a reduced path between two flags $F$ and $G$ is uniquely
  determined up to equivalence.
\end{prop}
\begin{proof}
  If $u$ and $v$ are both reduced and there are two flag paths
  $F\lmto{u}G$ and $F\lmto{v}G$ connecting $F$ and $G$, composing them
  we get a weak path $F -F$ with word $u\cdot v\inv$.  Corollary
  \ref{C:minimal_path_is_reduced} yields a strong reduct $w$ of
  $u\cdot v\inv$ with $F\lmto{w}F$. Corollary
  \ref{C:geschl_flaggenweg} implies that $w=1$ and thus $u\approx v$
  by Corollary \ref{C:word_inverses}.
\end{proof}

If $u$ is reduced, we will sometimes refer to $F\lmto{u}G$ by saying
that \emph{the reduced word $u$ connects $F$ to
$G$}.

\begin{lemma}\label{L:flags_im_pfad}
  Let $M$ be simply connected and $P$ be a reduced flag path in
  $M$. Denote by $A$ the set of vertices of flags occurring in $P$.
Every flag contained in $A$ appears in some permutation of $P$.
\end{lemma}
\begin{proof}
  We use induction on the length of $P$. Let $u=v\cdot s$ be the word
  of $P$ with $s=(l,r)$. Split $P$ in a path $Q$ from $F$ to $G$
with word $v$ and in the path from $G$ to $H$ with word $s$. Denote
by $B$ the vertices of flags occurring in $Q$. Consider
a flag $K\subset A$. If  $K\subset B$, then $K$ occurs in a 
permutation of $Q$ by induction. Thus, it occurs in a permutation of
$P$. If $K\nsubseteq B$, since $u$ is reduced, the letter $s$ does not
belong to the final segment of $v$, so by Lemma \ref{L:scaffold}
implies that the pair $(a_l,a_r)$ in $K$ is closed. Lemma
\ref{L:nice_entspannt} gives that $H$ is obtained by the
operation $\alpha_s$ to the nice set $B$. So $K\lmto{w}H$, where the
reduced word $w$ commutes with $s$. By Lemma \ref{L:Eq_flags}, there
is a unique $G'\subset B$ such that $G'\lmto{w}G$ and $G'\lmto{s}K$.
Induction gives that $G'$ is part of a reduced path $F\to G'\lmto{w}
G$, which is a permutation
of $Q$. Then $F\to G'\lmto{w} G\lmto{s}H$ is a permutation of $P$. We
  permute $w$ and $s$ and obtain $F\to G'\lmto{s} K\lmto{w}H$, as
  desired.
\end{proof}

Once the word of a flag path between $F$ and $G$ is fixed, the
intermediate flags appearing in the path are unique up to wobbling.
\begin{lemma}[Wobbling Lemma]\label{L:Wob}
  Given two paths between $F$ and $G$ with reduced word $s_1\cdots
  s_i\cdots s_n$,
\begin{figure}[!htbp]
\centering

\begin{tikzpicture}[>=latex,->]

\matrix (A) [matrix of math nodes,column
sep=1cm]
{  & H_1 & \cdots &  H_{n-1} &  \\
 F &   &  &  &   G, \\
   & H'_1 & \cdots & H'_{n-1} &   \\};

 \draw (A-2-1)  edge node[pos=.7, above left] {$s_1$} (A-1-2)
		edge node[pos=.7, below left] {$s_1$} (A-3-2) ;

\draw (A-1-2)  edge  (A-1-3) ;
\draw (A-3-2)  edge (A-3-3);

\draw (A-1-3)  edge  (A-1-4) ;
\draw (A-3-3)  edge (A-3-4);

 \draw[<-] (A-2-5)  edge node[pos=.7, above right] {$s_n$} (A-1-4)
		edge node[pos=.7, below right] {$s_n$} (A-3-4) ;

\end{tikzpicture}

\end{figure}

\noindent the flags $H_i$ and
$H'_i$ are equivalent modulo $\wob(s_1\cdots  s_i, s_{i+1}\cdots
s_n)$, for every $i$ in $\{1,\ldots,n-1\}$.
\end{lemma}
\begin{proof}
  Write $u=s_1\cdots s_i$ and $v=s_{i+1}\cdots s_n$. Suppose we are
given flags $H_i$ and $H_i'$ as in the previous picture. Hence
  $$F\lmto{u}H_i\lmto{v}G \qquad F\lmto{u}H'_i\lmto{v}G.$$

\noindent  Let $w$ be some reduced word with $H_i\lmto{w}H'_i$. By
Corollary \ref{C:minimal_path_is_reduced} and Proposition
\ref{P:unique_path_word},
the word $u$ is a strong reduct of $u\cdot w$. Likewise, the word
$v$ is a strong reduct of $w\inv\cdot v$. Proposition
\ref{P:wortwobbel} gives that $|w|\subset\wob(u,v)$, which yields the
 result.
\end{proof}

We finish this section by observing that  nice sets are
flag-connected.

\begin{prop}\label{P:nice_durch_flaggenwege}
  Let $M$ be simply connected and $A$ some union of flags from $M$.
  The set $A$ is nice \iff any two flags in $A$ can be connected by a
  reduced flag path which belongs to $A$.
\end{prop}
\begin{proof}
  Clearly, any union of flags satisfies that $A_a^b=A\cap M_a^b$.

Suppose it is nice. Consider two flags $F$ and $G$ in $A$ and connect
them in $M$ by some weak   path. Since $A$ is nice, we can find a weak
path $P$ belonging to
  $A$ which is reduced \emph{in the sense of $A$.}  In order to show
  that $P$ is a flag path (in the sense of $M$), we need only show
  that if $G$ is obtained from $F$ by a global application of
  $\alpha_s$ in $A$, then it remains a global application of
  $\alpha_s$ in $M$. Equivalently, for any $b$ in $G\setminus F$, if
  $\dd^A_s(b,F)=\infty$ then $\dd^M_s(b,F)=\infty.$ This is exactly
  the definition of niceness.

  \noindent Assume now that every two flags in $A$ are connected in
  $A$ by a reduced flag path. Consider two vertices $b$ and $c$ in
  $\A_s(A)$ with finite $s$-distance in $M$ and choose two flags $F$
  and $G$ in $A$ containing $b$ and $c$, respectively. Lemma
  \ref{L:abstand_flaggenpfad} (with $t=[0,N])$ and Lemma
  \ref{L:stark_aus_schwach} imply that we can connect $F$ and $G$ by a
  reduced path $P$ with word $u$ whose letters do not contain $s$. By
  assumption, there is a reduced flag path $P'$ in $A$ connecting $F$
  and $G$ as well. Thus, the word of $P'$ is a permutation of $u$ by
  Proposition \ref{P:unique_path_word}. So, again by Lemma
  \ref{L:abstand_flaggenpfad}, the points $b$ and $c$ are
  $s$-connected in $A$ and hence $A$ is nice.
\end{proof}

\section{Forking in the free pseudospace}\label{S:Forking}
In this section we provide a detailed description of nonforking over
nice sets and canonical bases. In particular, we obtain weak
elimination of imaginaries. The theory $\psn$ has trivial forking and
is totally trivial, as in \cite{BP00}.

We will work inside a sufficiently saturated model $M$. We start with
an easy observation which follows immediately from Theorem
\ref{T:extnice}.
\begin{prop}\label{P:stable}
 The theory $\psn$ is $\omega$-stable.
\end{prop}

\begin{proof}
  Work over a countable subset $A$, which we may assume to be nice.
  Theorem \ref{T:extnice} shows that every $1$-type over $A$ lies in
  some nice set $B$, obtained from $A$ by a finite number
  of applications $\alpha_s$. In
  particular, there are countably many quantifier-free types of such
  $B$'s over $A$ and thus countably many types by Corollary
  \ref{C:complete}. The theory $\psn$ is therefore $\omega$-stable.
\end{proof}

The following result will allow us to determine the
type of a flag over a nice set.

\begin{prop}\label{P:fusspunkt}
  Let $X$ be a nice set and $F$ a flag which is connected
  to a flag $G$ in $X$ by a reduced flag path $P$ with word $u$.
  The following are equivalent:
  \renewcommand{\theenumi}{\alph{enumi}}
  \begin{enumerate}
  \item\label{P:fusspunkt:indep} Let $v$ by a reduced word connecting
    $G$ to another flag $G'$ in $X$. Then $F$ is connected to $G'$
    by the reduct of $u\cdot v$.
  \item\label{P:fusspunkt:kleinstes} $u$ is the $\preceq$-smallest
    word connecting $F$ to a flag in $X$.
  \item\label{P:fusspunkt:minimal} $u$ is $\preceq$-minimal among words
    connecting $F$ to a flag in $X$.
  \end{enumerate}
\end{prop}
\begin{proof}
  (\ref{P:fusspunkt:indep})$\to$(\ref{P:fusspunkt:kleinstes}) follows
  from Corollary \ref{C:u_kleinerals_uv}.\\

  \noindent
  (\ref{P:fusspunkt:kleinstes})$\to$(\ref{P:fusspunkt:minimal})
  is trivial.\\

  \noindent (\ref{P:fusspunkt:minimal})$\to$(\ref{P:fusspunkt:indep}):
  Let $G'$ be any flag in $X$. Then $G$  is connected to $G'$ by a flag
  path $P$ with
  word $v$. By Proposition  \ref{P:nice_durch_flaggenwege},
we may assume that $P$ in $X$. Choose a decomposition $u=u_1\cdot
u'\cdot w$ and $w\cdot v'\cdot v_1=v$ as in Corollary \ref{C:amador},
with corresponding paths
  \[F\lmto{u_1\cdot u'}F^*\lmto{w}G\lmto{w}G^*\lmto{v'\cdot v_1}G',\]
  where $G^*$ is a flag in $X$.

  Let $b$ be a strong reduct of $w\cdot w$ connecting $F^*$ to $G^*$.
If $b\not\approx w$, consider the reduced word $c$ which connects
  $F$ with $G^*$. Since $c$ is a strong reduct of $u_1\cdot u'\cdot
b$, we have $c\preceq u_1\cdot u'\cdot b\prec u$, a contradiction. 
So
  $b$ is equivalent to $w$. We obtain a path from
  $F$ to $G'$ with word $u_1\cdot u'\cdot w\cdot v'\cdot v_1$. Up to
permutation, its only possible  strong reduct  is
$u_1\cdot w\cdot v_1$. So $F$ connects to $G'$ by word $u_1\cdot
w\cdot v_1$, which is the reduct of $u\cdot v$.

\end{proof}

\begin{definition}\label{D:basept}
  Given a nice set $X$. We call a flag $G$ in $X$ a \emph{base-point}
  of $F$ over $X$ if the conditions of Proposition \ref{P:fusspunkt}
  hold: The word connecting $F$ to $G$ is $\preceq$-minimal among
  words which connect $F$ with flags in $X$.
\end{definition}

\begin{lemma}\label{L:alpha_kette}
  Let $X$ be a nice set and $F_0\lmto{s_1}\cdots\lmto{s_n}F_n$ be a
  reduced flag path with $F_n\in X$. Then $F_n$ is a basepoint of
  $F_0$ over $X$ if and only if the flag $F_{i-1}$ is obtained from
  $F_i\cup\ldots F_{n}\cup X$ by a global application of
  $\alpha_{s_i}$ for all $i\ge 1$.

  In particular, if $F_n$ is a basepoint of $F_0$ over $X$, then
  $F_0\cup\ldots F_{n}\cup X$ is nice.
\end{lemma}
\begin{proof}
  The equivalence for $n=1$ is clear, since $F_0$ is obtained by an
  global application of $\alpha_{s_1}$ from $F_1\cup X=X$ if and only
  if there is no connection of $F_0$ to $X$ by a product of proper
  subletters of $s$ by Lemma \ref{L:abstand_flaggenpfad}.

  Proceed now by induction over $n$ and assume first that each
$F_{i-1}$ is obtained from $F_i\cup\ldots
  F_{n}\cup X$ by a global application of $\alpha_{s_i}$. Lemma
\ref{L:nice}  implies
  that $Y=F_1\cup\ldots F_{n}\cup X$ is nice. Furthermore,
the flag $F_1$ is a basepoint of $F_0$ over $Y$.  We will show that
property  \ref{P:fusspunkt} (\ref{P:fusspunkt:indep}) holds for $F_0$
and $F_n$ over $X$. Let $G$ be a flag in $X$. Choose reduced words
$x$, $y$ and $v$ with
\begin{alignat*}{2}
F_0&\lmto{x}G \text{ , }  F_1&\lmto{y}G \text{ and }
F_n&\lmto{v}G.
\end{alignat*}

Then $x$ is the
  reduct of $s_1\cdot y$ and, by induction, the word $y$ is the reduct
of  $s_2\cdots s_n\cdot v$. So $x$ is the reduct of $s_1\cdots
s_n\cdot v$. Therefore, the flag $F_n$ is a basepoint of
$F_0$ over $X$. 

For the other direction, note first that $F_{n-1}$ is obtained from
$F_n\cup X=X$ by a global application of $\alpha_{s_n}$. So
$Y=F_{n-1}\cup F_n\cup X$ is nice. If we can show  that $F_{n-1}$ is a
basepoint of $F_0$ over $Y$, we can conclude by induction. For that,
we will
verify \ref{P:fusspunkt}(\ref{P:fusspunkt:kleinstes}). Consider any
flag $G$ in $Y$ and let $x$ be the reduced word which connects $F_0$
to $G$.  If $G$ belongs to $X$, we have $s_1\cdots s_{n-1}\prec
s_1\cdots s_n\preceq x$. Otherwise, there are a flag $G'$ in $X$ and
  a word $w$ commuting with $s_n$ such the following
  diagram holds:
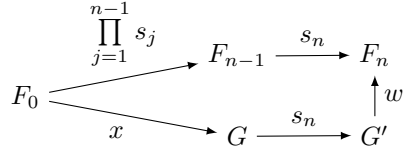
\begin{figure}[!htbp]
\centering

\begin{tikzpicture}[>=latex,->]

  \matrix (A) [matrix of math nodes,column
  sep=1cm]
          {  && F_{n-1} & F_n  \\
            F_0 & &   &  \\
            && G & G'   \\};

          \draw (A-2-1)  edge node[pos=.5, above ]
{$\prod\limits_{j=1}^{n-1} s_j$} (A-1-3)
		edge node[pos=.5, below left ] {$x$} (A-3-3) ;

                \draw (A-1-3)  edge node[pos=.5, above ] {$s_n$}
(A-1-4) ;
                \draw (A-3-3)  edge node[pos=.5, above ] {$s_n$}
(A-3-4) ;

                \draw (A-3-4)  edge node[pos=.5, right ] {$w$}
(A-1-4) ;

\end{tikzpicture}

\end{figure}

\noindent The reduced word $x'$ connecting $F_0$ with $G'$ is a
strong reduct of $x\cdot s_n$. Minimality of $u=s_1\cdots s_n$ yields
that $u\preceq x'$.  Corollary \ref{C:OrdnungMonoid} gives that
$s_1\cdots s_{n-1}\preceq x$.
\end{proof}

\begin{cor}\label{C:fusspunkttyp}
  Let $G$ be a flag in a nice set $X$. Given
a reduced word $u$, there is a flag $F$ a path $P$ from $F$ to $G$
with word $u$ such  that $G$ is the basepoint of $F$ over $X$. The set
$X\cup P$ is nice. The type of $F$ over $G$ (and thus, over $X$) is
uniquely  determined.
\end{cor}
\noindent Denote these types  by
\[\p_u(G)\text{  and  } \p_u(G)|X .\]

In order to describe the regular types and the dimensions of $\psn$,
we will need a characterisation of nonforking over nice sets in terms
of the reduction of the corresponding words connecting the
paths.

\begin{lemma}\label{L:forking}
  Let $F$ and $G$ be  flags, where $G$ lies in a nice set $X$. The
independence  $F\ind_G X$ holds if and only if $G$ is a basepoint of
$F$ over $X$.
\end{lemma}

\begin{proof}
  Let $u$ be the reduced word which connects $F$ to $G$. Then the type
  $\p_u(G)$ of $F$ over $G$ has a canonical extension $\p_u(G)|Y$ to
  every nice set $Y$ which contains $G$. Since $\psn$ is stable, it
  follows that $\p_u(G)|X$ is the only non-forking extension of
  $\p_u(G)$ to $X$.
\end{proof}

\begin{prop}\label{P:indep_for_flags}
  Given three flags with reduced paths $F\lmto{u}G$, $G\lmto{v}H$ and
  $F\lmto{w}H$, we have that $F\ind_G H$ if and only if $u\cdot v\to
w$.
\end{prop}
\begin{proof}
  If $F\ind_G H$, there is a nice set $X$ containing $G$ and $H$ such
  that $F\ind_G X$. But then $G$ is a basepoint of $F$ over $X$ and
  $u\cdot v\to w$ follows.

  Assume now $u\cdot v\to w$. Take $P$ the reduced path from $G$ to
  $H$ with word $v$. The set $P$ is nice. Enough to show $F\ind_G P$
  by verifying \ref{P:fusspunkt}(\ref{P:fusspunkt:indep}).  Given any
  flag $G'$ in $P$, by Lemma \ref{L:flags_im_pfad}, we may assume that
  $G'$ occurs in $P$.  Thus, write $v_1\cdot v_2=v$ with
  $G\lmto{v_1}G'\lmto{v_2}H$. If $x$ is reduced with $F\lmto{x}G'$,
  then
  \[u\cdot v=(u\cdot v_1)\cdot v_2\str x\cdot
  v_2\str w.\] \noindent By assumption $u\cdot v\to w$, so
Proposition
  \ref{P:strongvszerfallos} yields that no splitting occurs in the
  strong reductions above. This implies that $u\cdot
  v_1\to x$, which completes the proof.
\end{proof}

Note that the previous proof also yields $x\cdot v_2\to w$, which will
be used in the proof of Lemma \ref{L:redWegtrans}. Furthermore, we
have the following:

\begin{cor}\label{C:unabPfad}
Given flags $F$, $G$ and $H$ with $F\ind_GH$, then 
$$F\ind_GP,$$
where $P$ is the reduced flag path connecting $G$ to $H$.
\end{cor}

We will now compute the Morley rank $\Mr(p)$ and Lascar
rank $\U(p)$ of certain types in $\psn$.
\begin{definition}\label{D:Rdiv}
  Given reduced words $u$ and $v$, we say that $u$ is a proper
  left-divisor of $v$ if $u\not\approx v$ and there is a reduced $w$
such that $uw=v$ in  $\cox$.
\end{definition}
\noindent Note that $uw=v$ in  $\cox$ is equivalent to $u\cdot w \to
v$.

If $u$ is a proper left-divisor of $v$, it follows by Corollary
\ref{C:u_kleinerals_uv} that $u\prec v$. In particular, Lemma
\ref{L:kleiner_fundiert} yields that being a proper left-divisor is
well-founded. Let $\Rd$ be its foundation
rank and likewise let $\Rkl$ denote the foundation rank with respect
to $\prec$. 

\begin{lemma}\label{L:URdiv}
  For every flag $G$ and every reduced word $u$, 
\[\U(\p_u(G))=\Rd(u).\]
\end{lemma}

\begin{proof}
  Show $\U(\p_u(G))\leq\Rd(u)$ by induction on $\Rd(u)$. Assume 
that $\alpha < \U(\p_u(G))$. Then there is is a nice extension $X$ of
$G$ and a realisation $F$ of $\p_u(G)$ such that
$\alpha\leq\U(F/X)$. Since $F\nind_GX$, the type of $F$ over $X$ is
of the form $\p_v(H)|X$ for a reduced word $v$ and some flag $H$ in
$X$. Proposition  \ref{P:fusspunkt} $(\ref{P:fusspunkt:indep})$ and
Lemma
  \ref{L:forking} imply that $v$ is a proper left-divisor of
  $u$. By induction, we have
  \[\alpha\leq\U(F/X)=\U(\p_v(H))=\Rd(v)<\Rd(u),\]
  which proves $\U(\p_u(G))\leq\Rd(u)$.

  For the other direction, assume $\alpha<\Rd(u)$. Then there is a
  proper left-divisor $v$ of $u$ such that
$\alpha\leq \Rd(v)$. Choose
  a reduced word $w$ such that $v\cdot w\to u$. It is easy to
construct a flag $H$ with
  \[ F \lmto{v} H \lmto{w} G.\]
  Actually, such an $H$ exists whenever $v\cdot w\str u$. By
  Proposition \ref{P:indep_for_flags} we have $F\ind_HG$. Let $P$ be a
  path from $H$ to $G$ with associated word $w$. Seen as a
collection of points, the path $P$ is nice by Lemma \ref{L:scaffold}.
Corollary \ref{C:unabPfad} gives that  $F\ind_H P$,
so $\tp(F/P)=\p_v(H)|P$ and thus $F\nind_G P$. By induction,
\[\alpha\leq\Rd(v)= \U(\p_v(H)) <\U(\p_u(G).\]
\end{proof}

\begin{lemma}\label{L:MRRkl}
  For every flag $G$ and reduced word $u$, we have that
  \[\Mr(\p_u(G)) \leq \Rkl(u).\]
\end{lemma}

\begin{proof}
  Extend $\p_u(G)$ to $p=\p_u(G)|X$, where $X$ is an
$\omega$-saturated model containg $G$. The type $p$ contains a formula
$\varphi(x)$ stating that there is a weak path connecting the
flag $x$ to $G$ with word $u$. If $F$ realizes $\varphi$, then either
$F$ realizes $p$ or there is a path connecting $F$ to $X$ with 
word $\prec$-smaller that  $u$. For the latter, 
induction  gives that the Morley rank of $F$ over $X$ is
strictly smaller than $\Rkl(u)$. Since $X$ is $\omega$-saturated, this
implies that $\Mr(p)\leq\Rkl(u)$.
\end{proof}

\begin{lemma}\label{L:RdRkl}
  If $u=s_1\cdots s_n$ is reduced and $|s_i|\geq |s_{i+1}|$ for
  $i=1,\ldots,n-1$, then \[\Rd(u)=\Rkl(u) = \omega^{|s_1|-1}
  +\dotsb+\omega^{|s_n|-1}.\]
\end{lemma}

\begin{proof}
  Let $\ord$ be the function introduced in the proof of Lemma
  \ref{L:kleiner_fundiert}. Recall that for any reduced word $w$
  \[\Rd(w)\leq\Rkl(w) \leq \ord(w).\]
  If $u$ satifies the above hypotheses, then
$\ord(u)=\omega^{|s_1|-1} +\dotsb+\omega^{|s_n|-1}$. Hence, we need
only that $\ord(u)\leq\Rd(u)$. By induction, it is enough to find, for
every $\alpha<\ord(u)$, a proper left-divisor $u'$ of $u$ satisfying
the hypotheses of the Lemma such that $\alpha\leq\ord(u')$.

  There are two cases: If $|s_n|=1$, 
set $u'=s_1\cdots s_{n-1}$. If $|s_n|>1$, let $k$ be large enough such
that
  \[\alpha\leq\omega^{|s_1|-1}
  +\dotsb+\omega^{|s_{n-1}|-1}+\omega^{|s_n|-2}\cdot k\] Then choose
an appropriate sequence $t_1\cdots t_k$ of subletters of $s_n$, each
of size $|s_n|-1$, such that $u'=s_1\cdots s_{n-1}\cdot t_1\cdots t_k$
is
  reduced.
\end{proof}

\begin{cor}\label{C:U=MR}
  For every flag $G$ and every reduced word $u=s_1\cdots
s_n$ with $|s_i|\geq |s_{i+1}|$ for
  $i=1,\ldots,n-1$,  
  \[\U(\p_u(G))=\Mr(\p_u(G))=\omega^{|s_1|-1}
  +\dotsb+\omega^{|s_n|-1}.\]
\end{cor}
However, Lascar and Morley rank may differ in general, as the
following example  shows. 
\begin{remark}
  Consider the word $u=[0,1][1,3]$. It is easy to see that
  $\Rd(u)=\omega^2$ and $\Rkl(u)=\omega^2+\omega$, since the
inversion antiautomorphism $u \to u\inv$ preserves $\prec$. In
particular, the Lascar rank of $\p_u(G)$ is $\omega^2$. To compute the
Morley rank of $\p_u(G)$, consider the following sequence of words
\[u_k=\underbrace{[1][0]\cdots[1][0]}_k[1,3].\] The
  Morley rank of $u_k$ is at least $\Rd(u_k)=\omega^2$.  Since
  $\p_u(G)$ is the limit of the types $\p_{u_k}(G)$, its Morley rank
  of $\p_u(G)$ is at least $\omega^2+1$. Actually, it is easy to show
that $\Mr(\p_u(G))=\omega^2+1$. 
\end{remark}

The non-orthogonality classes of regular types over a nice set in
$\psn$ are given by global operations of $\alpha_s$ for $s$ varying
among all intervals. These types have trivial forking and therefore so
does $\psn$.

\begin{theorem}\label{T:rankregulartypes}
  The theory $\psn$ is $\omega$-stable of rank $\omega^N$.  Every type
  over a nice set $X$ is non-orthogonal to some type $\p_s(G)|X$,
where $G$ lies in $X$.  Forking is trivial, that is, any three
pairwise independent tuples are independent (as a set).
\end{theorem}

\begin{proof}
  By Lemma \ref{L:MRRkl}, the Morley rank of a flag cannot exceed
  $\Rkl([0,N])=\omega^N = \U(\p_{[0,N]}(G))=\Mr(\p_{[0,N]}(G))$, by
Corollary \ref{C:U=MR}. Thus, the Lascar and Morley
  rank of a flag over the emptyset are both $\omega^N$. Let $a$ be a
  vertex of $F$.  Lascar inequalities implie that
  $\U(F/a)+\U(a)\leq\U(F)$. Since $\U(a)>0$, this implies that
  $\U(a)=\omega^N$, and therefore $\Mr(a)=\omega^N$.

  Given a type $p$ over $X$, we may assume it is the type of a flag
  $F$ and thus determined by some reduced word $u$ connecting $F$ a
  basepoint $G$ over $X$. In particular, take any $s$ in the final
  segment of $u$. The type $p$ is hence non-orthogonal to the type
  $\p_s(G)|X$, since the connecting word of $F$ over the nice set
  consisting of $G$ together with a realisation of $\p_s(G)|X$ is
  $\prec$-smaller than $u$.

  Since the type $\p_s(G)$ has monomial Lascar rank, it is regular. A
  different way to see this is by taking a non-forking realisation $F$
  of $\p_s(G)|X$ and a forking realisation $F'$ to $X$. Now, since
  $F'$ forks with $X$ over $G$, Proposition
  \ref{P:fusspunkt}(\ref{P:fusspunkt:kleinstes}) gives a flag $G'$ in
  $X$ such that the word connecting $F'$ to $G'$ is a finite product
  $x$ of proper subletters of $s$. Since the reduction $s\cdot x\str
  s$ involves no splitting, the flags $F$ and $F'$ are independent
  over $G$ by Proposition \ref{P:indep_for_flags}.  The type $\p_s(G)$
  is regular, and so is $\p_s(G)|X$.

  Note that the geometry on every type $\p_s(G)$ is trivial: given
  three pairwise independent realisations $F_1$, $F_2$ and $F_3$ of
  $\p_s(G)$, note that any flag in $G\cup F_2\cup F_3$ must be either
  $G$, $F_2$ or $F_3$, for there are no new $s$-connections between
  them. Hence, \[ F_1 \ind_G F_2\cup F_3\]

  \noindent and forking is trivial on each $\p_s(G)|X$. Since the
  theory is superstable, forking is trivial \cite[Proposition
    2]{jBG91}.
\end{proof}

Nice sets are algebraically closed in $\psn^\mathrm{eq}$.

\begin{remark}\label{C:aclq_nice}
  Let $X$ be nice and $F$ be a flag with $F/A\in\aclq(X)$ for some set
  $A\subset [0,N]$. Then, the class $F/A$ lies in $X^\mathrm{eq}$.
That is, all vertices of $F$ with level outside $A$ belong to $X$.
\end{remark}
\noindent Since $X$ is nice, this is equivalent to $F/A=G/A$ for some
$G$ in $X$.
\begin{proof}
  Let $u$ be the reduced word connecting $F$ to a basepoint $G$ over
  $X$. By taking a sufficiently large initial segment of a sequence of
  $X$-independent realisations of $\tp(F/X)$, since the class $F/A$ is
  algebraic, we may find another realisation $F'$ with $F\ind_G F'$
  and $F/A=F'/A$. By Lemmata \ref{L:Eq_flags} and
  \ref{L:stark_aus_schwach}, there is a path connecting $F$ and $F'$
  whose reduced word $v$ satisfies $|v|\subset A$. Proposition
\ref{P:indep_for_flags} and the independence
  $F\ind_G F'$ imply  that $v$ is the reduct of $u\cdot u\inv$. Thus
$|u|=|u\cdot u\inv|=|v|\subset
  A$.  In particular, the flags $F$ and $G$ are equivalent modulo $A$.
\end{proof}

Let us now explicitly describe canonical bases of types over nice
sets. They are interdefinable with finite sets of real elements and
hence $\psn$ has weak elimination of imaginaries (\emph{cf.}
Corollary \ref{C:WEI}).

\begin{theorem}\label{T:Canbasis}
  Let $u$ be a reduced word and $G$ a flag. Then the canonical base of
  $\p_u(G)$ is interdefinable with $G/\sr(u)$.
\end{theorem}
Observe that  $G/\sr(u)$ is interdefinable with a finite set by
Definition \ref{D:Eq}.
\begin{proof}
  We have to show that $\p_u(G)$ and $\p_u(G')$ have a common nonforking
  extension if and only if $G$ and $G'$ are equivalent modulo
  $\sr(u)$. Or, in other words, given a nice set $X$, if $F$ is a
realisation of $\p_u(G)|X$, then $G'\in X$ is a basepoint of $F$ over
$X$ if and only if $G/\sr(u)=G'/\sr(u)$.

  If $v$ is a reduced word connecting $G$ and $G'$, then
  $G/\sr(u)=G'/\sr(u)$ means that $|v|\subset\sr(u)$, or equivalently
by Lemma \ref{L:v_in_sl}, that $v$ is right-absorbed by $u$.  Let $w$
be the reduced word connecting $F$ to $G'$. Then $w$ is the reduct of
$u\cdot v$ by Proposition \ref{P:fusspunkt}(\ref{P:fusspunkt:indep}).
The flag $G'$ is a basepoint of $F$ if an only if $w\approx u$. By
Corollary \ref{C:absorb_cox}, this is equivalent to $v$ being
right-absorbed by $u$.
\end{proof}

The following result will be useful in order to prove that the
theory $\psn$ is not $(N+1)$-ample.

\begin{lemma}[Basepoint Lemma]\label{L:Fusspunkt}
  Let $X$ be a nice set and $F$ connected by a reduced word $u$ to
its  basepoint $G$ in $X$. Assume $u=w\cdot v$ and pick a flag $H$
with
  \[ F \lmto{w} H \lmto{v} G.\]
  If $H/A\in X$ for some set $A\subset[0,N]$, then $|v|$ is
a subset of $A$.
\end{lemma}

\begin{proof}
  By Remark \ref{C:aclq_nice} and Corollary
\ref{C:minimal_path_is_reduced}, there is a flag
$G'$ in $X$ connected to $H$ by a reduced word $|v'|\subset A$.
The flag $G$ is a basepoint of $H$ over $X$ by Lemma
\ref{L:alpha_kette}. Proposition \ref{P:fusspunkt}
(\ref{P:fusspunkt:kleinstes})
gives that  $v\preceq v'$ and therefore $|v|\subset|v'|\subset A$.
\end{proof}

We finish the section with a strengthening of triviality, called
\emph{totally trivial} \cite{jBG91}, that is, given any set of
parameters $X$ and tuples $a$, $b$ and $c$ such that $a$ is both
independent from $b$ and $c$ over $X$, then it is independent from
$\{b,c\}$ over $X$. For theories of finite U-rank, both notions agree
 \cite[Proposition 5]{jBG91}.

By Lemma \ref{L:forking}, recall that, given a nice set $X$
and a distinguished flag $F_0$ in $X$, the following are equivalent
for any flag $F$,

 \begin{itemize}
  \item $F\ind_{F_0} X$
  \item $F\ind_{F_0} H $ for every flag $H$ in $X$
  \item $F_0$ is a basepoint of $F$ over $X$.
  \end{itemize}

Whilst considering flag paths, there is a simpler version of
transitivity of nonforking, due to the nature of the reduction with
non splitting.

\begin{lemma}\label{L:redWegtrans}
  Given flags $H$, $F$, $H_0$ and $F_0$, then $F\ind_{F_0} H_0$ and
  $F\ind_{H_0} H$ imply $F\ind_{F_0} H$. If there is a reduced path
  $F_0\lmto{v}H_0\lmto{w}H$, the converse also holds: $F\ind_{F_0} H$
  implies $F\ind_{F_0} H_0$ and $F\ind_{H_0} H$.
\end{lemma}
Observe that the condition on the path being reduced is needed for the
converse, as the following example shows, where $t\subsetneq s$:

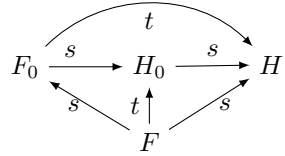
\begin{figure}[!htbp]
\centering

\begin{tikzpicture}[>=latex,->]

\matrix (A) [matrix of math nodes,column
sep=1cm]
{ F_0 & H_0 & H  \\[5mm]
   & F &  \\ };

 \draw (A-1-1)  edge node[pos=.5, above left] {$s$} (A-1-2) ;

 \draw[->] (A-1-1) to [out=45,in=135] node[pos=.5, below  ]
{$t$} (A-1-3) ;

\draw (A-1-2)  edge node[pos=.5, above ] {$s$} (A-1-3) ;

\draw (A-2-2)  edge node[pos=.5, left ] {$s$} (A-1-1)
		 edge node[pos=.5, left ] {$t$} (A-1-2)
		  edge node[pos=.5, right ] {$s$} (A-1-3);

\end{tikzpicture}

\end{figure}
\noindent Although $F\ind_{F_0} H$, since no splitting occurs when
reducing $s\cdot t$ to $s$, we have that $F\nind_{F_0} H_0$, as $t$ is
not the reduct of $s\cdot s$.
\begin{proof}
  We will use throughout the proof the characterisation of
  independence between flags given by Proposition
  \ref{P:indep_for_flags}. It actually follows from the proof of
  Proposition \ref{P:indep_for_flags} that the above converse holds,
  by taking $F$, $G$, $G'$, $H$ instead of $H$, $F$, $H_0$, $F_0$ in
  the proof. Alternatively, we may argue as follows: as $H_0$ occurs
  in a reduced path $P$ from $F_0$ to $H$, the proof of Proposition
  \ref{P:indep_for_flags} shows that $F\ind_{F_0} P$. This implies
  $F\ind_{F_0} H_0$. Since $F_0\lmto{v}H_0\lmto{w}H$, we have that
  $F_0\ind_{H_0}{H}$ by Proposition \ref{P:indep_for_flags}. This,
  together with $F\ind_{F_0}H$, the first part of the lemma and
  forking symmetry implies $F\ind_{H_0} H$.

  Assume now $F\ind_{F_0} H_0$ and $F\ind_{H_0} H$. Choose reduced
  paths $F\lmto{u}F_0$, $F_0\lmto{v}H_0$, $H_0\lmto{w}H$ and
  $F_0\lmto{x}H$. The word $a$ which connects $F$ to $H_0$ is the
reduct of $u\cdot v$. Also, the word $b$
  connecting $F$ to $H$ is the reduct of $u'\cdot w$. Hence, the word
$b$ is the reduct of $u\cdot v\cdot w$.  If $x$ were the reduct of
$v\cdot w$, then $b$ is the reduct of $u\cdot x$, so we are done.
Therefore, suppose that splitting occurs in
$v\cdot w\str x$. Treat first the case  $v=w=s$. Then $x$ is a
product of proper subintervals of $s$. By the Decomposition Lemma
  \ref{T:einfach_amador}, either $s$ is right absorbed by $u$, or
  $u=u_1\cdot u'$, where $u'$ is properly absorbed by $s$ and
  $u_1\cdot s$ is reduced. In the first case, the word $x$ is properly
absorbed by $u$, hence $F\ind_{F_0} H$.

  For the second case, decompose  $u=u_1\cdot u'$ as above. Then
$b$ (the  word connecting $F$ and $H$) equals $u_1\cdot s$. This
cannot be a strong reduct of $u_1\cdot u'\cdot x$, since the latter
is $\prec$-smaller, contradicting Proposition
\ref{P:strongvszerfallos}.

  For the general case, as in the proof of Proposition
\ref{P:strongvszerfallos}).we may assume that the splitting in
$v\cdot w\str x$ happens at the first step of the reduction.
Write hence $v=v'\cdot
  s$ and $w=s\cdot w'$, where $$F_0 \lmto{v'} K_1 \lmto{s} H_0\lmto{s}
  K_2 \lmto{w'} H.$$ The word $y$ connecting $K_1$ and $K_2$ consists
  of proper subletters of $s$.  By the first part of the proof, since
  $F\ind_{F_0} H_0$, we have that $F\ind_{F_0} K_1$ and $F\ind_{K_1}
  H_0$. Similarly, we obtain $F\ind_{H_0} K_2$ and $F\ind_{K_2} H$.
  By the previous discussion, we have that $F\ind_{K_1} K_2$. This,
  together with $F\ind_{F_0} K_1$, yields $F\ind_{F_0} K_2$, by
  induction on the length of $v$. Now, the word connecting $F_0
  \lmto{} K_2$ is a strong reduction of $v'\cdot y$, so $\prec$-smaller
  than $v$. Induction on the complexity of $v$ together with
  $F\ind_{K_2} H$ gives $ F\ind_{F_0} H$, as desired.
\end{proof}

In order to prove the total triviality of $\psn$, we will use the
following lemma, a stronger form of which follows already
from total triviality, without the assumption $F_0\ind_AB$, since
if $$A\lmto{s}B\lmto{t}C,$$
\noindent where $s$ and $t$ commute with each other, then $B$ is
definable in $A\cup C$, by Lemma \ref{L:Wob}.

\begin{lemma}\label{L:tot_triv_special}
  Let $A$, $B$, $C$, $F$, $F_0$ be flags and $s$ and $t$ two commuting
  letters, such that $A\lmto{s}B\lmto{t}C$. If the following
  independencies hold:
\begin{alignat*}{2}
F&\ind_{F_0}A \text{ , }  F&\ind_{F_0}C \text{ and }  F_0&\ind_AB,
\end{alignat*}
\noindent then  $F\ind_{F_0}B$.
\end{lemma}
\begin{proof}

 In order to show that $F\ind_{F_0}B$, since $F\ind_{F_0}A$, by Lemma
\ref{L:redWegtrans}, we need only show $F\ind_AB$.
Thus, consider a reduced word $z$  with
$F\lmto{z}B$ and connect the above flags by reduced paths as in
the diagram below.

\begin{figure}[!htbp]
\centering
\begin{tikzpicture}[>=latex,->]

  \matrix (A) [matrix of math nodes,column
  sep=2cm, row sep=5mm]
          {  & B &   \\
            A &   & C \\
            & F_0&    \\[3mm]
	    & F & \\};

         \draw (A-4-2)  edge node[pos=.5, left ]
{$u$} (A-3-2)
  edge node[pos=.5, below right ]
{$b$} (A-2-3)  edge node[pos=.5, below left ]
{$a$} (A-2-1) ;

  \draw (A-3-2)  edge node[pos=.5, above left ]
{$y$} (A-2-3)  edge node[pos=.5, above right ]
{$x$} (A-2-1) edge node[pos=.5, right ]
{$v$} (A-1-2);

 \draw (A-2-3)  edge node[pos=.5, above right ]
{$t$} (A-1-2);

\draw (A-2-1)  edge node[pos=.5, above left ]
{$s$} (A-1-2);

\end{tikzpicture}

\end{figure}

  Assume for a contradiction that $F\nind_AB$. Then $z$, which is a
strong reduct of $a\cdot s$, is not the reduct of $a\cdot s$. This
has two consequences: first, the letter $s$ does not occur in the
final segment of $z$. Secondly, up to permutation,  the path
$F\lmto{a}A$ ends with a flag $A'\lmto{s} A$, such that  $A'$ is
connected to $B$ by a word consisting of proper subletters of $s$.
Since $F_0\ind_AB$, such a flag $A'$ cannot occur in any  permutation
of $x$. Thus, as $a$ is a reduct of $u\cdot x$, it follows that $s$
commutes with $x$ and is in the final segment of $u$. In particular,
the word $x\cdot s$ is reduced, which implies that $v$ is (up to
permutation) the word $x\cdot s$.

  On the other hand, the word $v=x\cdot s$ is a strong reduct of
$y\cdot t$. It is easy to see that this can only be possible if (after
permutation) $y$ has the form $y'\cdot s$,where $y'$ and $s$
commute. The independence $F\ind_{F_0}C$ implies that
$b$ is the reduct of $u\cdot y$. Hence $s$ still
belongs to the final segment of $b$. Finally, since $z$ is a strong
reduct of $b\cdot t$, the word $s$ must belong to the final segment of
$z$, which contradicts that $F\nind_AB$.
\end{proof}

In order to ensure the independence of a flag with respect to a
whole flag path over a nice set, it is enough to check the
independence with respect to the set itself and the end flag of the
path.

\begin{lemma}\label{L:unab_von_Pfad}
  Let $A$ be a nice set and a reduced path $P$ connecting a flag $H$
  to a basepoint in $A$.  Given a flag $F_0$ in
  $A$ and a flag $F$, we have that $F\ind_{F_0} A\cup P$ if and only
  if $F\ind_{F_0} A$ and $F\ind_{F_0} H$.
\end{lemma}

\begin{proof}
  Left-to-right is clear. Assume now that $F\ind_{F_0} A$ and
  $F\ind_{F_0} H$. Since $A\cup P$ is nice by Lemma
  \ref{L:alpha_kette}, in order to check that $F\ind_{F_0} A\cup P$,
  we need to check that $F\ind_{F_0} H'$ for any flag $H'$ in $A\cup
  P$ by the remark above Lemma \ref{L:redWegtrans}. This is clear for
flags in $A$, so let $H'$ be in $A\cup P$ but not in $A$.

  We treat first the case where $H'$ is in $P$. Let $H_0$ be the
  base-point of $H$ in $A$. We have then that $F_0\ind_{H_0}H$ and
  $F\ind_{F_0}H$ by assumption, which implies $F\ind_{H_0}H$ by Lemma
  \ref{L:redWegtrans}. Since the path $P$ is reduced, Lemma
  \ref{L:redWegtrans} gives $F\ind_{H_0}H'$, which together with
  $F\ind_{F_0} H_0$ implies $F\ind_{F_0} H'$.

  For the general case, we will proceed by induction on the length of
$P$, based on the above paragraph.  Thus, it suffices to consider the
case where $P$ has length $1$ and let $s$ be its letter:

$$H_0\lmto{s}H.$$

 If $H'$ is a flag in $A\cup P$ not completely contained in $A$, it
differs from $H$ only on the indices outside $s$.  As in the proof of
Lemma \ref{L:flags_im_pfad}, we can find a reduced word $w$ commuting
with $s$ such that $H'\lmto{w}H$. Furthermore, there is some flag
$H'_0$ in $A$ with $H'_0\lmto{w}H_0$ and $H'_0\lmto{s}H'$.

  Note that $H'_0$ is again a basepoint of $H'$ over $A$, so in
particular $F_0\ind_{H'_0}H'$. By induction on the length of $w$, we
may assume that $w$ is a letter $t$. Setting $A=H'_0$, $B=H'$ and
$C=H$, the hypotheses of Lemma \ref{L:tot_triv_special} are
satisfied. We conclude that $F\ind_{F_0}H'$, which gives the desired
result.
\end{proof}

We now have all the ingredients to prove total triviality of
forking.

\begin{prop}\label{P:totally_trivial}
  The theory $\psn$ is totally trivial, that is, given  any set
of parameters $X$ and tuples $a$, $b$ and $c$ such that $a$ is both
independent from $b$ and $c$ over $X$, then it is independent from
$\{b,c\}$ over $X$. In particular, the canonical base of a tuple is
the union of the canonical bases of each singleton.
\end{prop}

\begin{proof}
 We may assume that our parameter set $X$ is nice, by
choosing a small model containing it independent from $a,b,c$.

Suppose first that the tuples $a$, $b$ and $c$ consists of singletons:
By transitivity, choose flags $H_1$ and $H_2$ independently from $a$
over $X$ containing $b$ and $c$ respectively. Choose now a flag $F$
containing $a$ independently from $H_1$ and from $H_2$ over $X$. We
need only to show that $$F\ind_X H_1 \cup H_2.$$

Let $F_0$ and $H_0$ be basepoints of $F$ and $H_1$ respectively over
$X$. Since $F\ind_{F_0} X$ and $F\ind_{X} H_1$, we have that
$F\ind_{F_0} X\cup P_1$ by Lemma \ref{L:unab_von_Pfad}, where $P_1$
denotes the reduced flag path (connecting $H_1$ to $H_0$) determined
by $H_1$ over $X$. The set $X\cup P_1$ is again nice by Lemma
\ref{L:alpha_kette}. Work now over $X\cup P_1$ in order to show that
$F\ind_{F_0} X\cup P_1\cup P_2$, where $P_2$ is the flag path given by
$H_2$ over $X\cup P_1$. Lemma \ref{L:unab_von_Pfad} gives that $F$ is
independent from $H_1\cup H_2$ over $X$.

Transitivity of forking allows us to work with finite tuples by
choosing accordingly nonforking extensions for each coordinate. The
result now follows by local character.
\end{proof}
Since $\psn$ is superstable, \cite[Proposition 7]{jBG91} allows to
conclude the following.
\begin{cor}\label{C:perfect}
  The theory $\psn$ is perfectly trivial, that is, given given any set
of parameters $X$ and tuples $a$, $b$ and $c$ such that $a$ and $b$
are both independent over $X$, then so are they  over $X\cup\{c\}$.
\end{cor}

\begin{cor}\label{C:WEI}
  The theory $\psn$ has weak elimination of imaginaries.
\end{cor}

\begin{proof}
 By Proposition \ref{P:totally_trivial}, in order to study the
canonical base of a real tuple $\bar a$ over an algebraically closed
set $B$ (in $\psn^\mathrm{eq}$), we may assume that $\bar a$ is an
enumeration of a flag $F$. Furthermore, we may suppose that $B$ is
nice. By Theorem \ref{T:Canbasis}, the canonical base is
interdefinable with a finite set, thus we get weak elimination of
imaginaries.
\end{proof}

Although the theory $\psn$ is not $1$-based, being $N$-ample by
Proposition \ref{P:ample}, it is $2$-based, i.e. the canonical base
of a type is determined by two independent realisations.

\begin{prop}\label{P:2based}
 Let $u$ be a reduced word and $X$ a nice set. The  canonical base
of $\p_u(G)|X$ is algebraic over two independent realisations.
\end{prop}

\begin{proof}
 Let $F$ and $F'$ be realisations of $\p_u(G)|X$, which are
$X$-independent. Since the base-point is only determined up to
$\sr(u)$-equivalence, pick a common base-point $G$ in $X$ for
both $F$ and $F'$.

As $F\ind_X F'$ and $F\ind_G X$, combining Lemmas
\ref{L:redWegtrans} and \ref{L:unab_von_Pfad}, we conclude that
$F\ind_G F'$. Therefore, the word connecting $F$ and $F'$ is the
reduction of $u\cdot u\inv$. Write $u=u_1 \tilde{u}$, where
$\tilde{u}$ is the final segment of $u$. Hence,

\[ u\cdot u\inv \to u_1 \cdot \tilde{u} \cdot u_1\inv,\]

\noindent as the diagram shows:
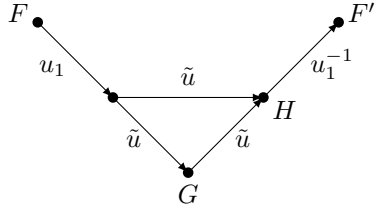
\begin{figure}[h]
\centering

\begin{tikzpicture}[>=latex,text height=.25ex,text depth=0.25ex]

\fill node (0,0) (G) {} circle (2pt);
\fill (1,1) node (H) {} circle (2pt);
\fill (-1,1) circle (2pt);
\fill (2,2)  node [right](F') {$F'$} circle (2pt);
\fill (-2,2) node [left] (F) {$F$} circle (2pt);
 \node (H1) [below right of=H, node distance=4mm] {$H$};
 \node (G1) [below of=G, node distance=4mm] {$G$};
%
 \draw[->] (0,0) -- (1,1) node[pos=.5, below right] {$\tilde u$};
 \draw[->] (1,1) -- (2,2) node[pos=.5, below right] {$u_1\inv $} ;

  \draw[->]  (-2,2) -- (-1,1) node[pos=.5, below left] {$u_1$};
 \draw[->] (-1,1) --(0,0) node[pos=.5, below left] {$\tilde u$} ;
 \draw[->] (-1,1)  -- (1,1) node[pos=.5, above] {$\tilde u$} ;
\end{tikzpicture}
\end{figure}

Note that $G$ and $H$ are equivalent modulo $|w|\subset \sr(u)$. By
Lemma \ref{L:Wob}, the flag $H$ is determined by $F$ and $F'$ modulo
$\sr(u)\cap \sL(u_1\inv)$ and thus, modulo $\sr(u)$. In particular,
the canonical base $G/\sr(u)$ is algebraic over $F,F'$.

\end{proof}

\section{Ample yet not wide ample}\label{S:Beweis}
This last section shows that the ample hierarchy defined in \ref{D:CM}
is proper, since the theory of the free $N$-dimensional pseudospace
$\psn$ is $N$-ample but not $(N+1)$-ample. We will furthermore show
that it is $N$-tight with respect to the family $\Sigma$ of Lascar
rank $1$ types, if $N\geq 2$.

The proof that $\psn$ is $N$-ample is a direct translation of the
proof exhibited in \cite{BP00}, which we nontheless include for the
sake of the presentation.

\begin{prop}\label{P:ample}
Consider a flag
$a_0-\cdots-a_N$. We have the following:
\renewcommand{\theenumi}{\alph{enumi}}
\begin{enumerate}
\item  $\aclq(a_0,\ldots,a_i)\cap\aclq(a_0,\ldots,a_{i-1},a_{i+1})=
\aclq(a_0,\ldots,a_{i-1})$ for every $0\leq i<N$.
\item $a_{i+1} \ind_{a_i} a_0,\ldots, a_{i-1}$ for every $1\leq i<N$.
\item $a_N \nind a_0$.
\end{enumerate}

In particular, the theory $\psn$ is $N$-ample.
\end{prop}
\begin{proof}
  In order to prove $(a)$, fix some $i<N$ and choose parameters
  $b_i,\ldots,b_N$ independently from $a_i, a_{i+1}$ such
  that $$a_0-\cdots-a_{i-1}-b_i-\cdots-b_N$$ is a flag. Set
  $X=\{a_0,\ldots,a_{i-1},b_i,\ldots,b_N\}$, which is nice.

  By Fact \ref{F:Inter}, assume for a contradiction that there is an
  element $e$ in $$\aclq(X,a_i)\cap
  \aclq(X,a_{i+1})\setminus\aclq(X).$$ Choose now $a'_i$ realising
  $\tp(a_i/X,e)$. Since the element $e$ lies also in
  $\aclq(X,a'_i)$, then $a_i\nind_{X} a'_i$. As the
  $\preceq$-minimal word connecting $a_i$ (or rather, the flag
  $a_0-\cdots-a_N$) to $X$ is $[i,N]$, it follows from Lemma
  \ref{L:forking} that $a_i$ and $a'_i$ (or rather, generic flags
  containing them) are connected through a finite product of proper
  intervals of $[i,N]$. Compactness (and Lemma
  \ref{L:abstand_flaggenpfad}) implies that there exists a natural
  number $n$ such that

  $$\tp(a_i/X,e)\models \dd_{[i,N]}(x, a_i)\leq n.$$

  Let $m$ be such that $2m>n$. Consider the reduced word
  \[u= \underbrace{[i+1,N] \cdot i \cdots [i+1,N]
    \cdot i}_{2m}.\] Corollary \ref{C:fusspunkttyp} provides us with a
  flag $F$ and a path $P$ from $G=a_0-\cdots-a_N$ to $F$ with word $u$
  \[F=F_0\lmto{[i+1,N]}F'_0\lmto{i}F_1\lmto{[i+1,N]}\dotsb\lmto{{[i+1,N]}}
  F'_{m-1}\lmto{i} F_m=G\] such that $G$ is the basepoint
  of $F$ over the nice set $G$. Since the $F_i$ and $F'_i$ are
  connected by the word $[i,N]$ to $G$, they have all the same type
over $X$. Denote
  \begin{align*}
    F_r&=a_0-\cdots-a_{i-1}-a^r_i-a^r_{i+1}-\cdots-a^r_N\\
    F'_r&=a_0-\cdots-a_{i-1}-a^r_i-a^{r+1}_{i+1}-\cdots-a^{r+1}_N.
  \end{align*}
  Since $F_0$ and $F'_0$ have the same type over $X$, they have also
  the same type over $Xa^0_i$ and therefore over $Xe$. This implies
  that $e$ belongs to $\aclq(Xa^1_{i+1})$. Similarly, the flags $F'_0$
and $F_1$ have the same type over $Xa^1_{i+1}$ and therefore over
$Xe$, which implies that $e$ belongs to $\aclq(Xa^1_i)$. Iterating, we
see that $a^m_i$ has the the same type over $Xe$ as
  $a_i$. This implies that $\dd_{[i,N]}(a^m_i, a_i)\leq n$, which
gives a contradiction since the shortest path between $a_i$ and
  $a^m_i$ in $\A_{[0,N]}$ is
$$a^0_i-a^1_{i+1}-a^1_i-\dotsb-a^m_{i+1}-a^m_i,$$
\noindent of length $2m$.\\

  For $(b)$, chose generic flags $F$ containing $a_{i+1}$ and
  $G$ containing $a_0,\ldots,a_i$. The canonical
  base $\cb(a_{i+1}/ a_0,\ldots,a_i)$ equals $\cb(F/G)$. On the other
  hand, the flags $F$ and $G$ are connected by the reduced word
  $u=[0,i][i+1,N]$. So
  \[\cb(F/G)=G/\sr(u)=G/\bigl([0,i-1]\cup[i+1,N]\bigr)=a_i\]
  by Theorem \ref{T:Canbasis}, which gives the desired independence.\\

  For $(c)$, choose a generic flag $F$ which contains $a_N$ and
  a generic flag $G$ which contains $a_0$. Then $\cb(a_N/a_0)$ equals
  $\cb(F/G)$. On the other hand the reduced word connecting $F$ to $G$ is
  $u=[0,N-1][1,N]$, So
  \[\cb(F/G)=G/\sr(u)=G/[1,N]=a_0,\]
  which is clearly not algebraic over $a_1$. Thus,
  \[a_N \nind a_0.\]

\end{proof}

Before the proof that $\psn$ is not $(N+1)$-ample, we need
some auxiliary results on the nature of the reduced words
 arising from the hypothesis on ampleness.

\begin{lemma}\label{L:v1komm}
  Consider nice sets $A$ and $B$ and a flag $F$ such that
  $\aclq(AB)\cap\aclq(A,F)= \aclq(A)$ and $F\ind_B A$. Let $u = u_B$
  (resp.\ $u_A$) be the $\preceq$-minimal word connecting $F$ to a
  flag $G_B$ in $B$ (resp.\ $G_A$ in $A$) and let $v$ be the reduced
  word connecting $G_B$ to $G_A$. If
 \begin{align*}
    u&=u_1\cdot u',&
    v'\cdot v_1&=v
  \end{align*}
 is the fine decomposition as in Theorem \ref{T:einfach_amador}, then
 $v_1$ is commuting.
\end{lemma}

\begin{proof}
  By hypothesis, $F\ind_{G_B}G_A$, so the product $u_1\cdot v_1$ is
  equivalent to $u_A$. Suppose for a contradiction that $v_1$ is not
  commuting. Hence, we may decompose $v_1=v_1^1\cdot s \cdot v_1^2$,
  where $v_1^2$ is the final segment of $v_1$ and $s$ does not commute
  with $v_1^2$.

  By Lemma \ref{L:schlucktheorie}, we can write $u'=u'_2\cdot u'_1$,
  where $u'_1$ is left-absorbed by $v_1^1\cdot s$, the word $u'_2$
  commutes with $v_1^1\cdot s$ and is left-absorbed by $v_1^2$. We
  have the following diagram:

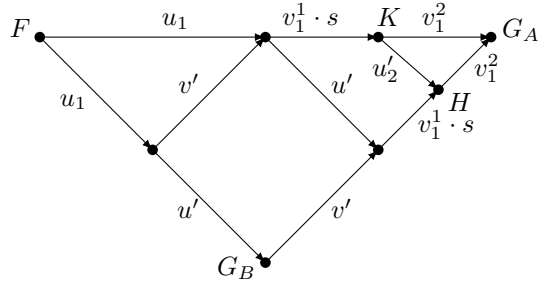
\begin{figure}[h]
\centering

\begin{tikzpicture}[>=latex,text height=.25ex,text depth=0.25ex]

\fill node [below left] (0,0) (FB) {$G_B$} circle (2pt);
\fill (1.5,1.5) circle (2pt);
\fill (-1.5,1.5) circle (2pt);
\fill (2.3,2.3)  node (H) {} circle (2pt);
\fill (0,3)  circle (2pt);
\fill (1.5,3) node (K) {}  circle (2pt);
\fill (3,3)  node (FA) {} circle (2pt);
\fill (-3,3) node [left] (F) {$F$} circle (2pt);

\node (H1) [below right of=H, node distance=4mm] {$H$};
\node (K1) [above right of=K, node distance=2mm] {$K$};
\node (FA1) [right of=FA, node distance=4mm] {$G_A$};

\draw[->] (0,0) -- (1.5,1.5) node[pos=.5, below right] {$v'$};
 \draw[->] (1.5,1.5) -- (2.3,2.3) node[pos=.5, below right]
{$v_1^1\cdot s$} ;
 \draw[->] (2.3,2.3) -- (3,3) node[pos=.5, below right] {$v_1^2$}  ;
\draw[->] (-3,3)  -- (-1.5,1.5) node[pos=.5, below left] {$u_1$} ;
 \draw[->] (-1.5,1.5) -- (0,0) node[pos=.5, below left] {$u'$} ;
\draw[->] (-3,3) -- (0,3) node[pos=.5, above right] {$u_1$};
 \draw[->] (0,3) -- (1.5,3) node[pos=.4, above] {$v_1^1\cdot s$};
 \draw[->] (1.5,3) -- (3,3) node[pos=.5, above] {$v_1^2$};
\draw[->] (1.5,3) -- (2.3,2.3) node[pos=.5, below left] {$u'_2$};
\draw[<-] (1.5,1.5) -- (0,3) node[pos=.5, right] {$u'$};
\draw[->] (-1.5,1.5) -- (0,3) node[pos=.5, left] {$v'$};
\end{tikzpicture}
\end{figure}

\noindent  where the path connecting $K$ and $H$ is given by $u'_2$.
So the flags $H$ and $K$ are equivalent modulo $|u'_2|$.

Lemma \ref{L:Wobschluck} gives that $\wob( v'\cdot v_1^1\cdot s,
v_1^2)$, the wobbling of $v$ at $H$, is contained in $W=\wob(u_1\cdot
v_1^1\cdot s, v_1^2)$. In particular, by Lemma \ref{L:Wob}, the class
$H/W$ lies in $\aclq(AB)$. So does $K/(|u'_2|\cup W)$, which also lies
$\aclq(AF)$. By assumption, $K/(|u'_2|\cup W)$ lies in $\aclq(A)$
since $\aclq(AB)\cap \aclq(AF)=\aclq(A)$, and therefore in $A$ by
Remark \ref{C:aclq_nice}.  Since $u_A$ is $\preceq$-minimal
connecting $F$ to a flag in $A$, Lemma \ref{L:Fusspunkt} implies

$$|v_1^2| \subset |u'_2| \cup W.$$

Observe that $u'_2$ centralises $s$ and $W$ is contained in $s \cup
\Cent(s)$. Hence, so does $|v_1^2|$.  Since $v_1$ is reduced and
$v_1^2$ is commuting, no letter of $v_1^2$ is contained in $s$. So
$v_1^2$ must commute with $s$, which contradicts the definition of
$v_1^2$.

\end{proof}

\begin{prop}\label{P:Endstueckwirdgroesser}
  Consider nice sets $A$ and $B$ and a flag $F$ such that
  $\aclq(AB)\cap\aclq(A,F)= \aclq(A)$ and $F\ind_B A$. Let $u = u_B$
   (resp.\ $u_A$) be the minimal word connecting $F$ to a flag $G_B$
in $B$  (resp.\ \ $G_A$ in $A$) (These
are the same  hypotheses as in Lemma \ref{L:v1komm}). Then, either
$F\ind_{A\cap B} AB$ or $u$ is nontrivial and its final segment
$\tilde u$, as a set of indices, is strictly contained in $\tilde
u_A$, the final segment of $u_A$.

  In particular, consider the reduced word $v$ which connects $G_B$ to
$G_A$ and the associated fine decomposition
\begin{align*}
    u&=u_1\cdot u',&
    v'\cdot v_1&=v,
  \end{align*} 
 as in Theorem \ref{T:einfach_amador}. If $$F\nind\limits_{A\cap B}
A,$$ 
\noindent then $\tilde u$ is nontrivial and 

$$|v'|  \nsubseteq |\tilde u| \subsetneq |\tilde u_A|.$$

\end{prop}
\begin{proof}
  Since $F \ind_B A$ and $v$ is reduced connecting $G_B$
to $G_A$, the word $u\cdot v$ reduces to $u_A$.  If 
\begin{align*}
    u&=u_1\cdot u'&
    v'\cdot v_1&=v.
  \end{align*}
\noindent is the fine decomposition (\emph{cf.} Theorem
\ref{T:einfach_amador}) applied to $u$ and $v$, we may thus assume
that $u_A= u_1\cdot v_1$.

Let $H$ be  the flag in the path $G_B\lmto{v}G_A$ between $ v'$ and
  $v_1$. Likewise, let $K$ be the flag in the path $F\lmto{u_A}G_A$
  between $u_1$ and $v_1$. Note that $H$ and $K$ are connected through
  $u'$. Furthermore, Lemma \ref{L:Wobschluck} gives that
  $\wob(v',v_1)$ is contained in $W=\wob(u_1, v_1)$. Since $H$ and $K$
  are equivalent modulo $|u'|$ and $H/\wob( v',v_1)$ lies in
  $\aclq(AB)$ by Lemma \ref{L:Wob}, it follows that $K/(W\cup |u'|)$
  lies in $\aclq(AB)\cap \aclq(AF)=\aclq(A)$ and whence in $A$ by
Remark \ref{C:aclq_nice}. Lemma
  \ref{L:Fusspunkt} gives now

  $$|v_1| \subset |u'| \cup W.$$

  \noindent Decompose the final segment of $u$ as
  \[\tilde u=w_1\cdot w_2,\] where $w_2$ is the final
  segment of $u'$ and $w_1$ is a subword of the final segment of
  $u_1$.  In particular $u'=u''\cdot w_2$ and $w_1$ and $u''$
  commute. We
show first that $w_1$ and $v_1$ commute: since
  $u'\subset \Cent(w_1)$ and $W\subset \sr(u_1)\subset
  |w_1|\cup\Cent(w_1)$, we have $v_1\subset |w_1|\cup\Cent(w_1)$. A
  letter $s$ of $v_1$ cannot be contained in $|w_1|$, since $u_1\cdot
  v_1$ is reduced. So $s$ belongs to $\Cent(w_1)$, which gives the
desired result. Recall that $v_1$ is commuting by Lemma
\ref{L:v1komm}. Thus, the final segment of $u_A=u_1\cdot v_1$ is
  \[\tilde u_A=w_1\cdot v_1,\]
  \noindent which clearly contains $\tilde u$, as $|w_2|$ is a subset
of $|v_1|$.

Suppose the inclusion is not strict. Hence, we have $|w_2|=|v_1|$.
Then  $|v_1|\subset \sr(u)$ and hence $|v|\subset \sr(u)$.  So $G_B$
and $G_A$ are equivalent modulo $\sr(u)$. In particular, the
canonical  base $\cb(F/B)$ lies in $A$ and thus
  $$F\ind_{A\cap B} B.$$
  \noindent Since $F\ind_B A$, transitivity of non-forking implies
  that $F\ind_{A\cap B} AB$.

  Finally, assume that $\tilde u=1$, which forces $u=1$ and thus
$v'=1$. In particular, since $|v_1|\subset |\tilde u_A| \subset
\sr(u_A)$ and $G_A$ and $G_B$ are equivalent modulo $v=v_1$, they are
equivalent modulo $\sr(u_A)$, so $\cb(F/A)= G_A/\sr(u_A)$ lies in $B$
and hence $F\ind_{A\cap B}  A$.

Similarly, if $|v'|  \subset |\tilde u| \subset |\tilde u_A| \subset
\sr(u_A)$ , we conclude as before that $\cb(F/A)= G_A/\sr(u_A)$ lies
in $B$ and thus $F\ind_{A\cap B}  A$.

\end{proof}

We can now state and prove the desired result.

\begin{theorem}\label{T:NOTample}
 The theory $\psn$ is not $(N+1)$-ample and is $N$-tight with respect
to the family of Lascar rank $1$ types.
\end{theorem}

\begin{proof}

  By Remark \ref{R:prof}, we need only show that given tuples
  $b_0,\ldots,b_{N+1}$ with: \renewcommand{\theenumi}{\alph{enumi}}
  \begin{enumerate}
  \item $\aclq(b_i,b_{i+1})\cap\aclq(b_i,b_{N+1})=\aclq(b_i)$ for every
    $0\leq i<N$.
  \item $b_{N+1} \ind_{b_i} b_{i-1}$ for every $1\leq i\leq N$,
  \end{enumerate}

  \noindent then there is some $i$ in $\{0,\ldots,N-1\}$ such that
  \[b_{N+1} \ind_{\aclq(b_i)\,\cap\,\aclq(b_{i+1}) } b_i.\] By Fact
  \ref{F:Inter}, it suffices to prove this for tuples $b_0,\ldots
  b_N$ which enumerate small models $B_0,\ldots B_N$, although for
the proof, we only require that each $B_i$ is nice.  Total triviality
  (\emph{cf.}\ Proposition \ref{P:totally_trivial}) allows us to
  assume that $b_{N+1}$ consists of a single flag $F$.

  Choose for every $i\leq N$ a basepoint $F_i$ for $F$ over $B_i$.
  Note that we obtain the following configuration:

  
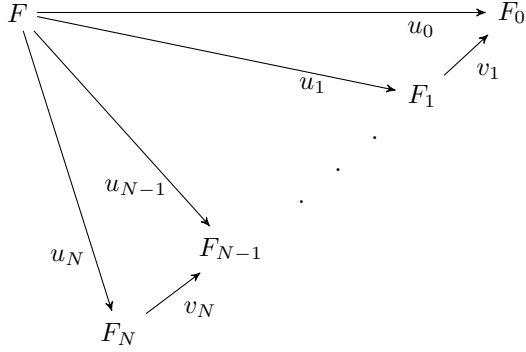
\begin{figure}[!htbp]
    \centering

\begin{tikzpicture}[->,>=stealth',shorten >=1pt]

  \node (F0) {$F_0$};

  \node (F1) [below left=8mm of F0] {$F_1$} ;
  \node (ldot1) [below left=2mm of F1] {$.$} ;
  \node (ldot2) [below left=2mm of ldot1] {$.$} ;
  \node (ldot3) [below left=2mm of ldot2] {$.$} ;
  \node (FN1) [below left=3mm of ldot3] {$F_{N-1}$} ;
  \node (FN) [below left=8mm of FN1] {$F_N$} ;
  \node (F) [left=6cm  of F0] {$F$} ;

  \path
  (F) edge node[pos=.8,left] {$u_N$} (FN)
  edge node [pos=.8, left] {$u_{N-1}$} (FN1)
  edge node [pos=.7, below right] {$u_1$} (F1)
  edge node [pos=.8, below right] {$u_0$} (F0);

  \path (FN) edge node [pos=.5, below right] {$v_N$} (FN1);
  \path (F1) edge node [pos=.5, below right] {$v_1$} (F0);

\end{tikzpicture}

  \end{figure}

  \noindent such that $u_i\cdot v_i$ reduces to $u_{i-1}$, for every
  $i$ in $\{1,\ldots, N\}$, due to $(b)$. Proposition
  \ref{P:Endstueckwirdgroesser} implies that either, for some $i<N$,

  $$F \ind_{B_i\cap B_{i+1} } B_i,$$ or the final segment $\tilde
u_{i+1}$ of $u_{i+1}$ is non-trivial and strictly contained
  in $\tilde u_i$ for all $i<N$.

  The second possibility for every $i<N$ delivers a strictly
increasing sequence of length $N+1$ of non-empty subsets of
$\{0,\ldots,N\}$, which implies that $\tilde u_0$ equals $[0,N]$ and
thus $u_0=[0,N]$. Hence

\[F\ind B_0,\]

\noindent and thus $$F\ind_{\aclq(B_0)\,\cap\,\aclq(B_1)} B_0.$$

The first possibility implies
  \[F \ind_{\aclq(B_i)\cap\aclq(B_{i+1})} B_i,\]
as desired. This proves that $\psn$ is not $(N+1)$-ample.\\

  Suppose now that $N\geq 2$. In order to show that $\psn$ is
  $N$-tight with respect to $\Sigma$, where $\Sigma$ denotes the
  collection of all Lascar rank $1$ types, assume we are given tuples
  $b_0,\ldots,b_N$ witnessing the following conditions:

\renewcommand{\theenumi}{\alph{enumi}}
\begin{enumerate}
 \item  $\aclq(b_0,\ldots,b_i)\cap\aclq(b_0,\ldots,b_{i-1},b_{i+1})=
\aclq(b_0,\ldots,b_{i-1})$ for every $0\leq i<N$.
\item\label{I:ind} $b_{i+1} \ind_{b_i} b_0,\ldots, b_{i-1}$ for every
$1\leq i<N$.
\suspend{enumerate}

As in Remark \ref{R:prof}, it follows that:

\resume{enumerate}
\item\label{I:inter} $\aclq(b_{i+1})\cap \aclq(b_i)\subset \aclq(b_0)$
for every $1\leq i<N$.
\item $b_N \ind_{b_i} b_{i-1}$ for every $1\leq i<N$.
\item\label{I:int} $\aclq(b_i,b_{i+1})\cap\aclq(b_i,b_N)= \aclq(b_i)$ for
every $0\leq i<N-1$.
\end{enumerate}

Note that (almost) internality is preserved under taking nonforking
restrictions. Furthermore, if a tuple $d$ is (almost) internal over
$C$ and $e$ is algebraic over $Cd$, then so is $e$ (almost)
internal over $C$. Thus, we may as before replace every $b_i$ by a
nice set $B_i$ by Fact \ref{F:Inter} and assume that $b_N$ is a flag
$F$ by total triviality (\emph{cf.} Proposition
\ref{P:totally_trivial}). In particular, we need to prove that
$\cb(F/B_0)$ is almost $\Sigma$-internal over $B_1$. 

As before, let $u_i$ be $\preceq$-minimal connecting $F$ to a flag
$F_i$ of $F$ in $B_i$ for $i< N$. Since $N\geq 2$, there is (at least)
one triangle to apply  Proposition \ref{P:Endstueckwirdgroesser} and
thus, either for some $0\leq i<N-1$ we have that $$F\ind_{B_i\cap
B_{i+1}}
B_i,$$ \noindent or the final segment $\tilde u_{i+1}$ of
$u_{i+1}$ is non-trivial and strictly contained in $\tilde u_i$ for
every $i< N$. The
independence $$F\ind_{B_i\cap B_{i+1}} B_i$$
\noindent  implies by properties  $(\ref{I:ind})$ and
$(\ref{I:inter})$ that $F\ind_{\aclq(B_0)\cap\aclq( B_1)} B_0.$
So $\cb(F/B_0)$ is algebraic over $B_1$, and hence internal over
$B_1$.

Otherwise, if $$F\nind_{B_i\cap B_{i+1}} B_i$$ \noindent for every
$i<N$, then the final segment $\tilde u_0$ must have length $N$.
Consider the fine decomposition $u_1=u_1^1\cdot u_1'$ and $v_1'\cdot
v_1^1=v_1$ from Theorem \ref{T:einfach_amador}. Proposition
\ref{P:Endstueckwirdgroesser} implies that $|v_1'|$ is not fully
contained in $\tilde u_1$, which must then have non-trivial
centraliser.  Since $\tilde u_1$ has size $N-1$, it must be either
$[2,N]$ or $[0,N-2]$. Let us consider the first case. The canonical
base $\cb(b_N/B_0)$ is $F_0$ modulo $\sr(u_0)=[1,N]$, which is the
$0$-vertex $f_0$ of $F_0$. Furthermore, since $v_1 =[0]\cdot [1,N]$,
the vertex $f_0$ is directly connected to $B_1$ and, by theorem
\ref{T:rankregulartypes}, it has rank $1$ over $B_1$, so the canonical
base $\cb(F/B_0)$ is $\Sigma$-internal over $B_1$, which concludes
the proof.

\end{proof}


\end{document}